\documentclass{aims} % Use the aims.cls file to compile your paper
\usepackage{amsmath}
\usepackage{paralist}
\usepackage[misc]{ifsym}
\usepackage{epsfig} %For pictures: screened artwork should be set up with an 85 or 100 line screen
\usepackage{epstopdf} %This is to transfer .eps figure to .pdf figure; please compile your article using PDFLaTex or PDFTeXify.
\usepackage[colorlinks=true]{hyperref}
\hypersetup{urlcolor=blue, citecolor=red}
\allowdisplaybreaks

\def\currentvolume{X}
 \def\currentissue{X}
  \def\currentyear{200X}
   \def\currentmonth{XX}
    \def\ppages{X-XX}
     \def\DOI{10.3934/xx.xxxxxxx}

\input{additionalSymbols.sty}
\newtheorem{theorem}{Theorem}[section]

\newtheorem{lemma}[theorem]{Lemma}
\newtheorem{proposition}[theorem]{Proposition}

\theoremstyle{definition}

\newtheorem{assumption}{Assumption}

\title[Optimal sensor placement in the WC4D-Var framework]
{Optimal sensor placement under model uncertainty in the weak-constraint 4D-Var framework} % Only the first word and proper nouns should be capitalized.

\author[Alen Alexanderian, Hugo D\'iaz, Vishwas Rao and Arvind K.\ Saibaba]{Alen Alexanderian, Hugo D\'iaz, Vishwas Rao and Arvind K.\ Saibaba}


\thanks{
The work of AA was supported in part by the US
National Science Foundation (NSF) grant DMS-2111044. 
HD and AKS were supported by the Department of Energy, Office of Science, Advanced Scientific Computing Research (ASCR) Program through the award DE-SC0023188. AKS was also supported by the NSF, in part, through the award DMS-1845406.  VR was  supported by the U.S. Department of Energy, Office of Science, ASCR Program under contracts DE-AC02-06CH11357 and DE-SC0023188.}

\begin{document}
\maketitle

% Enter the first author's name and email address; email addresses are required for each author.
% Use footnote notations to indicate address and affiliations with commas between numbers if more than one address applies; see below for examples.
\centerline{\scshape
Alen Alexanderian$^{{\href{mailto:aalexan3@ncsu.edu}{\textrm{\Letter}}}*1}$
Hugo D\'iaz$^{{\href{mailto:hsdiazno@ncsu.edu}{\textrm{\Letter}}}1}$
Vishwas Rao$^{{\href{mailto:vhebbur@anl.gov}{\textrm{\Letter}}}2}$
and Arvind K.\ Saibaba$^{{\href{mailto:asaibab@ncsu.edu}{\textrm{\Letter}}}1}$
}

\medskip

{\footnotesize
% Enter the full affiliation and country name:
% Do not insert commas or periods at the end of lines.
 \centerline{$^1$Department of Mathematics, North Carolina State University, USA}
} % Do not forget to end {\footnotesize with the sign }

\medskip

{\footnotesize
 % Enter the full affiliation and country name:
 \centerline{$^2$Argonne National Laboratory, USA}
}

\bigskip

% The name of the handling editor will be entered by AIMS production staff.
% "Communicated by Handling Editor" is not needed for special issue.
 \centerline{(Communicated by Handling Editor)}

%%%%%%%%%%%%%%%%%%%%%%%%%%%%%%%%%%%%%%%%%%%%%%%%%%%%%%%
%             5. ABSTRACT
%%%%%%%%%%%%%%%%%%%%%%%%%%%%%%%%%%%%%%%%%%%%%%%%%%%%%%%

\begin{abstract}
In data assimilation, the model may be subject to uncertainties and errors. The weak-constraint data assimilation framework enables incorporating model uncertainty in the dynamics of the governing equations.
We propose a new framework for near-optimal sensor placement in the weak-constrained setting. 
This is achieved by first deriving a design criterion based on the expected information gain, which involves the Kullback-Leibler divergence from the forecast prior to the posterior distribution.
An explicit formula for this criterion is provided, assuming that the model error and background are independent and Gaussian and the dynamics are linear.
We discuss algorithmic approaches to efficiently evaluate this criterion through randomized approximations.
 To provide further insight and flexibility in computations,
 we also provide alternative expressions for the criteria.
We provide an algorithm to find near-optimal experimental designs using column subset selection, including a randomized algorithm that avoids computing the adjoint of the forward operator.
Through numerical experiments in one and two spatial dimensions, we show the effectiveness of our proposed methods.
\end{abstract}

%%%%%%%%%%%%%%%%%%%%%%%%%%%%%%%%%%%%%%%%%%%%%%%%%%%%%%
%                   6. BODY
%%%%%%%%%%%%%%%%%%%%%%%%%%%%%%%%%%%%%%%%%%%%%%%%%%%%%%

% Only the first word and proper nouns of section titles should be capitalized.
% The title of section 1:
% \section{Introduction}
% 1. Introduction
%{\bf Some things, such as a group of assumptions, are not working! Hugo}
\section{Introduction}

Accurately estimating the state of complex dynamical systems is essential in fields such as meteorology, oceanography, and climate modeling. Data assimilation techniques, which combine noisy observations with mathematical models, play a critical role in reconstructing and predicting the state of these systems. These methods have been extensively studied and developed, with notable contributions in both theoretical and applied contexts and advancements in geophysical and climate-related applications \cite{1994Evensen, Freitag,GHIL1991141, Gratton, Xiang-Yu}. The Strong-Constraint Four-Dimensional Variational (SC4D-Var) method assumes perfect model dynamics. 
However, in many applications, the mathematical model does not perfectly represent the underlying physics, and a failure to account for this can be detrimental to the data assimilation results \cite{ngodock2017weak}.
The Weak-Constraint 4D-Var (WC4D-Var) method \cite{ Laloyaux,Tremolet, xu2007}   addresses these shortcomings by accounting for uncertainties in both the model and the observations.  This uncertainty is variously referred to as model error, model misspecification, or model uncertainty, depending on the field of study, though they all reflect  limitations or inaccuracies in the forward model. However, throughout this paper we refer to these as model errors.

In many applications of interest, the data is in the form of sensor measurements. A question related to data assimilation is: How should one optimally collect data under physical or budgetary constraints? This falls under the purview of optimal experimental design (OED). 
To enable optimal sensor placement, we assume that we have identified $n_s$ candidate sensor locations  and we have to choose $k$ optimal sensor locations. There are many challenges that need to be addressed to determine the optimal placement of these sensors.
 First, we need to consider a notion of optimality---a design criterion---that accounts for the model errors. 
Second, even a single evaluation of the design criterion is computationally expensive. And finally, exploring the space of  
$\binom{n_s}{k}$
possibilities  is computationally infeasible in practical settings.  

The goal of this article is to address all these aforementioned challenges. To the first point, while standard approaches for OED ignore model error, we develop a notion of optimality in the WC4D-Var framework and study the behavior of the optimality criterion as the model error vanishes. To the second point, we develop several efficient methods for evaluating the objective function, and to the third point,  we build on our recent work on OED using column subset selection~\cite{eswar2024bayesian} and adapt it to the specific criterion that we derive and time-dependent problems. The specific contributions of this article are listed next, followed by a discussion on related work in the literature.

\subsubsection*{Contributions}

This article introduces a novel framework for near-optimal sensor placement within the context of the WC4D-Var method. By explicitly incorporating model errors, we develop a method to optimize sensor locations for enhanced data assimilation. Our key contributions include:

\begin{enumerate}
    \item {\bf New optimality criterion}: In  \Cref{sec:eig}, we derive a novel OED criterion for the WC4D-Var framework that quantifies the expected information gain (EIG) from the forecast to the posterior distribution. We provide a closed-form expression for this criterion when the forecast and the posterior distributions are Gaussian, and which involves the log-determinant of the  prior and posterior covariance matrices.
    \item {\bf Convergence to standard SC4D-Var criterion}:
    We show that as the model  error vanishes, the optimality criterion for WC4D-Var converges to the criterion used in the standard SC4D-Var formulation, ensuring consistency between the methods, see \Cref{ssec:comparison}.
    
    \item {\bf Alternative forms of the criterion}: 
    In  \Cref{section:Alternative_formulations},  we derive alternative formulations of the WC4D-Var EIG criterion (each formulation is equivalent, up to an additive constant, independent of the data). Notably, all these formulations are interconnected through appropriate Schur complements and provide different computational benefits.
    \item{\bf Matrix-free method for evaluating the design criterion}:
    In \Cref{sec:preclanc}, 
    we propose a matrix-free approach for evaluating the WC4D-Var design criterion, based on stochastic trace estimators, and the Lanczos method to enable efficient computation of the different formulations of the criterion.

    \item {\bf Near-optimal sensor placement}: In \Cref{sec:sensor_placement}, we propose a novel approach for near-optimal sensor placement by combining ideas from low-rank tensor approximations with the column subset selection problem, specifically leveraging the Golub-Klema-Stewart (GKS) method.  We introduce a randomized variant of the GKS method that is computationally efficient and can be implemented in a matrix-free and adjoint-free method. This makes the methods applicable to a wide variety of scenarios, especially involving legacy implementations.

    \item {\bf Numerical experiments}: In \Cref{sec:Numerical_experiments}, we conduct comprehensive numerical experiments in both one and two spatial dimensions to evaluate the performance of our proposed methods. The one-dimensional setting enables a direct comparison with the optimal sensor placement, which is computationally intractable in higher dimensions due to its combinatorial nature. The numerical experiments show that the algorithms are accurate and computationally efficient.

\end{enumerate}
\subsubsection*{Related work}  
Model  errors arises in various contexts, including statistical inference, Bayesian inverse problems, regression analysis, and data assimilation. In this section, we first review approaches for handling model errors in general, then focus on strategies that specifically address model errors in the context of OED.
In statistical inference, model errors have been extensively studied. Specifically, within the framework of Maximum Likelihood Estimation (MLE), several studies have identified conditions under which the estimator converges to a well-defined limit, even when the assumed probability model is incorrect, see \cite{Halbert} and references therein. A related approach involves the use of the Generalized Bayes framework, or Gibbs posteriors, which can be viewed as a generalization of the data-generating process or likelihood. This approach accounts for model uncertainty by extending the standard Bayesian methodology \cite{Baek_2023}.

Finally, several approaches have been developed to integrate model errors into OED for Bayesian inverse problems. 
One such approach is the Bayesian Approximation Error (BAE), which addresses uncertainties and model discrepancies, such as those introduced by discretization \cite[Chapter 7]{Kaipio:1338003}. A related technique is pre-marginalization, which has been explored in \cite{alexanderian2024optimaldesignlargescalenonlinear} in the context of nonlinear Bayesian inverse problems under model uncertainty. 
For a broader review of OED under model errors, see \cite[Section 6.1]{Huan_Jagalur_Marzouk_2024}. 
These methods enhance model robustness and improve the reliability of experimental designs, aligning with the principles of robust OED \cite{Scarinci2021BayesianIU, ASPREY2002545, attia2023robustaoptimalexperimentaldesign}. Furthermore, within the framework of robust OED,  mixed models, such as those studied in \cite{Tsirpitzi2023}, have proven effective, particularly in regression models with dependent error processes \cite{SacksI, Dette}.

In data assimilation, model and observation errors are addressed using methods like the Ensemble Kalman Filter (EnKF) \cite{Tippett2003EnsembleSR}, WC4D-Var \cite{Freitag}, and hybrid 4DVar–EnKF approaches \cite{Dong, EnKFandHybridGain}.
However, to our knowledge, OED for the WC4D-Var framework has not been explored previously.
%

% 2. Background
\section{Background}\label{section:background}

Variational methods for data assimilation aim to integrate information about a dynamical system, usually represented by a PDE model, with observed data. 
We begin by describing the Strong-Constraint 4DVar (SC4D-Var) method in a Bayesian framework in \Cref{ssec:sc4d-var}. 
This method aims to recover the true state at the initial time denoted by $\bmu_0$ (inversion parameter). 
Then, in \Cref{ssec:wc4d-var}, we discuss the Weak-Constraint version (WC4D-Var), which will consider model error in the dynamical system.

\subsection{Strong constraint-4DVar}\label{ssec:sc4d-var}
In SC4D-Var, the state variables $\{\bmu_{\ell}\}_{\ell = 0}^{n_T}$, representing the system's states at discrete time steps, evolve sequentially starting from the initial time $\tau_0=0$. The time steps are ordered as  $\tau_0 \le \dots \le \tau_{n_T}$, and the states follow the discrete dynamical system:
\begin{align}
    \bmu_{\ell + 1} = \Map{\ell}{{\ell + 1}} (\bmu_\ell) \qquad  \mbox{for } 0 \leq \ell \leq n_T-1,
\end{align}
where $\Map{\ell}{{\ell + 1}}$ represents the operator evolving the state from time $\tau_\ell$ to $\tau_{\ell+1}$ and $\bmu_0\in \R^{d_S}$ is the initial condition. 
\begin{assumption}
\label{assumption:data}    
We assume that the data is collected according to 
\begin{align}
    \ylobs = & \> \O_{\ell }\left(\bmu_{\ell}\right) + \bm{r}_{\ell}  \qquad  \mbox{for all } 0 \leq \ell \leq n_T,
\end{align}
where $\bm r_{\ell} 
\sim \mathcal{N}(\bm 0, {\bm R_\ell})$ are independent and $\ylobs \in \R^{n_s}$ for $0\le \ell \le n_T$. Here $n_s$ is the number of sensors.
\end{assumption}

The goal of SC4D-Var is to recover the initial condition $\bmu_0$ from the discrete measurements 
$\{ \ylobs\}_{\ell = 0}^{n_T}$. 
As is the prevalent approach, the background distribution is assumed to be Gaussian. 
\begin{assumption} \label{assumption:background_distribution}
\textbf{Background distribution}:
The initial condition $\bmu_0$ is assumed to be described by our prior/background knowledge:
\begin{align}
\bmu_0  \sim \mu_{\rm back} := \mathcal{N}( \bmu_0^{b}, \Gb).\label{def:u_0_assumptions}
 \end{align}
\end{assumption}
An application of Bayes' rule gives the posterior distribution $\bm{u}_0| \bm{y}_{0}, \dots, \bm{y}_{n_T}$ with density $\pi(\bm{u}_0| \bm{y}_{0}, \dots ,\bm{y}_{n_T})$ and corresponding measure $\mu_{\rm post}^{\rm{sc}}$. Under Assumptions~\ref{assumption:data}-\ref{assumption:background_distribution}, the posterior density takes the form $\pi(\bm{u}_0| \bm{y}_{0}, \dots, \bm{y}_{n_T}) \propto \exp(- \mathcal{J}_{\text{SC}}(\bm{u}_0))$, where $\mathcal{J}_{\text{SC}}(\bm{u}_0)$ is given by:
\begin{align*}
\mathcal{J}_{\text{SC}}(\bmu_0) = & \>  \frac{1}{2} \left( \bm{u}_0 - \bm{u}^{b}_0 \right)^\top {\Gb^{-1}} \left( \bm{u}_0 - \bm{u}^{b}_0 \right) \\
  & \qquad + \frac{1}{2} 
  \sum_{\ell = 0}^{n_T} (\O_\ell(\bm{u}_\ell) - \ylobs)^\top \bm{R_\ell}^{-1} (\O_\ell(\bm{u}_\ell) - \ylobs).
\end{align*}  
 Note that we have used a non-standard approach by including the term for $\ell = 0$ to facilitate a direct comparison with the WC4D-Var formulation.
 The maximum a posteriori (MAP) estimate of this posterior can be found by solving the optimization problem:
\begin{align*}
\min_{\bmu_0} \mathcal{J}_{\text{SC}}(\bmu_0) &\quad \text{subject to:} \quad \bmu_{\ell + 1} = \Map{\ell}{{\ell + 1}} (\bmu_\ell), \quad 0 \le \ell \le n_T -1.
\end{align*}

To enable analytical expressions for subsequent analysis, we make the following assumption.
\begin{assumption}\label{assumption:forward_model}
    \textbf{Linearity}: We assume that the forward model $\Map{{\ell-1}}{\ell}$ for $1\le \ell \le n_T$ is linear or can be linearized around the background trajectory generated from $\bmu_0^b$. 
    We also assume that the observation operator $\O_\ell$ is linear for $0 \le \ell \le n_T$.
\end{assumption}
According to the assumptions made so far (Assumptions~\ref{assumption:data}, \ref{assumption:background_distribution}, and \ref{assumption:forward_model}), the posterior distribution is Gaussian and is denoted by $\mu_{\rm post}^{\rm sc} = \mathcal{N}(\bmu_{0,sc}^\text{post}, \bm{\Gamma}_{\text{sc,post}})$. The expressions for the posterior mean $\bmu_{0,sc}^\text{post}$ and the posterior covariance $\bm{\Gamma}_{\text{sc,post}}$ take the form 
\begin{align*}
    \bm{\Gamma}_{\text{sc,post}}^{-1} &= 
    {\Gb^{-1}}  
  +  \sum_{\ell = 1}^{n_T}
  (\O_\ell \Map{0}{{\ell}})\t 
  \bm{R_\ell}^{-1}
  (\O_\ell \Map{0}{{\ell}}),\\
 \bmu_{0,sc}^\text{post} &= \bm{\Gamma}_{\text{sc,post}} \left( \sum_{\ell = 1}^{n_T} (\O_\ell \Map{0}{{\ell}})\t\bm{R_\ell}^{-1} \ylobs
   +  \Gb^{-1} \bmu^{b}_0 \right).
   \end{align*}

\subsection{WC4D-Var data assimilation}\label{ssec:wc4d-var}
The WC-4DVar generalizes SC-4DVar by adding a model-error term to address discrepancies between observations and model forecasts. We assume that this error term is additive:
\begin{equation}
\begin{aligned}
    \bmu_{\ell+1} = & \> \Map{\ell}{{\ell+1}}\!\!(\bmu_\ell) + \bm\eta_{\ell+1} && \mbox{for } 0 \leq \ell \leq n_T-1,
\end{aligned} \label{eq:WC_system}
\end{equation}
where $\bm\eta_{\ell}$ represent the model errors.
From this point forward, as in SC4D-Var, we adopt Assumptions~\ref{assumption:data}, \ref{assumption:background_distribution}, and \ref{assumption:forward_model}. Additionally, we introduce the following assumption regarding the model error:
\begin{assumption}\label{assumption:model_measurement_error}
\textbf{Model and measurement errors}: We assume that the  model errors are Gaussian and are independent; that is, we assume: 
 \begin{align}
\begin{aligned}
\bm \eta_{\ell}  \sim 
\mathcal{N}(\bm 0,\bm Q_{\ell}) \quad  \mbox{ for }  \quad 1 \leq \ell \leq n_T. 
\end{aligned}\label{def:noise_assumptions}
\end{align}
\end{assumption}
The assumption of independence among vectors $\{\bm{\eta}_\ell\}$ is made for computational reasons, since it leads to block-diagonal covariance matrices.  We note that it is a common practice in the literature to assume that the model errors are Gaussian \cite{carrassi2018data, fisher2011weak}.
In particular, we define $\bm\eta = \bmat{\bm{\eta}\t_1 & \dots & \bm{\eta}\t_{n_T} }\t$ 
and $\bm{r} = \bmat{\bm{r}\t_0 & \dots & \bm{r}\t_{n_T}}\t$.

We emphasize that the Gaussianity and independence assumptions in \Cref{assumption:model_measurement_error} serve as a practical starting point for the development of scalable sensor placement methods in the WC4D-Var setting. These assumptions simplify the statistical model, making the optimization problems tractable. 
More expressive models of uncertainty, such as those arising from deep generative models or invertible neural networks \cite{pmlr-v97-behrmann19a}, could in principle better capture non-Gaussian, correlated errors. However, such methods typically demand more data and greater computational resources. Our present focus is thus on building a robust foundation under Gaussian assumptions, with a view toward incorporating more general modeling frameworks in future work.

By Assumptions \ref{assumption:data} and \ref{assumption:model_measurement_error}, we have  
$\bm\eta \sim \mathcal{N}
\left(\bm 0,\Gq \right)$ and 
$\bm r \sim 
\mathcal{N}\left(\bm 0,\Gmeas \right)$, where $$\Gq:=
\text{blkdiag}(
\bm Q_{1},\dots, \bm Q_{n_T})\quad \text{and }\quad  
\Gmeas:= \text{blkdiag}(\bm R_0,\dots, \bm R_{n_T}).$$
Similarly, we define $\bmu = \bmat{\bmu\t_0 & \dots & \bmu\t_{n_T} }\t$ and 
$\bmyo = \bmat{(\bmyo_0)\t & \dots&  (\bmyo_{n_T})\t}\t$.

Based on Assumptions~\ref{assumption:data},~\ref{assumption:background_distribution}, and~\ref{assumption:model_measurement_error}, the posterior distribution has density  $\pi(\bmu|\bmyo) \propto \exp(-\mathcal{J}_{\rm{WC}}(\bmu))$, where 
\begin{align}
\begin{aligned}\label{def:Jwc}
  \mathcal{J}_{\text{WC}}(\bmu) &=    \frac{1}{2} \left( \bm{u}_0 - \bm{u}^{b}_0 \right)^\top {\Gb^{-1}} \left( \bm{u}_0 - \bm{u}^{b}_0 \right) \\
  & \qquad + 
  \frac{1}{2} \sum_{\ell = 0}^{n_T} (\O_\ell(\bm{u}_\ell) - \ylobs)^\top \bm{R_\ell}^{-1} (\O_\ell(\bm{u}_\ell) - \ylobs) \\
  & \qquad+ 
  \frac12\sum_{\ell = 1}^{n_T} (\bm{u}_{\ell} -  \Map{\ell-1}{{\ell}}(\bmu_{\ell-1}))\t \bm{Q}_{\ell}^{-1} (\bm{u}_{\ell} -  \Map{\ell-1}{{\ell}}(\bmu_{\ell-1})).  
\end{aligned}\end{align}
Further simplification is possible if we include~\Cref{assumption:forward_model}. 
In particular, in this case, the posterior distribution becomes Gaussian. 
To derive the appropriate expressions, we first introduce additional notation, outline the forecast prior, and present the form of the posterior distribution.

\subsubsection*{Evolution operator} 
 We define the evolution operator
\begin{align}\label{eq:EvolProp}
    \Map{j}{\ell}:= \Map{j}{{j+1}}  \circ \cdots  \circ     \Map{{\ell-1}}{{\ell}} \quad \mbox{for } 0 \leq j \leq \ell  \leq n_T.
\end{align}
By convention,  $\Map{\ell}{\ell}$  is the identity matrix. 
We also define  
 \[\bm p : =  
    \bmat{
         \bm u_0\\
         \bm\eta
     }, \quad \text{where} \quad \bm \eta:=
    \begin{bmatrix}
        \bm \eta\t_1&\bm \eta\t_2& \ldots& \bm \eta\t_{n_T}
    \end{bmatrix}\t.\]
 We introduce a mapping $\L$ from the state space to the model error space
 as follows:
\begin{align}
\begin{aligned}
   \L:\R^{(n_T+1) d_S}& \rightarrow \R^{(n_T+1) d_S},\qquad \bmu&\mapsto  \L(\bmu)=\bm p,
\end{aligned} \label{eq:def_L}
\end{align}
where $d_S$ represents the spatial dimension. Additionally, we define the total number of degrees of freedom 
\begin{equation}
    \label{eqn:N}
    N_d := (n_T+1)\cdot d_S.
\end{equation}

The operator \(\L\), defined in \eqref{eq:def_L}, and its inverse can be expressed as follows in terms of the evolution operator $\Map{j}{\ell}$:
\begin{align}
\begin{aligned}
\L&=
\begin{pmatrix*}[c]
 \I          & \bm 0       & \bm 0  &\cdots  & \bm 0\\
 -\Map{0}{1} & \I          & \bm 0  &\cdots  &\bm 0\\
 \bm 0       & -\Map{1}{2} & \I &  &  \\
  \vdots     & \cdots      & \ddots  &  \ddots & \bm 0  \\
   \bm 0     & \cdots      & \bm 0  &  -\!\!\!\!\!\!\Map{n_T-1}{n_T} &  \I \\
\end{pmatrix*},\\
  \L^{-1} &=
  \begin{pmatrix*}[c]
    \I & & & &  \\
    \Map{0}{1}& \I & & & \\
    \Map{0}{2}& \Map{1}{2} & \I& &  \\
    \vdots& \vdots& &\ddots &  \\
    \Map{0}{n_T}& \Map{1}{n_T}& \cdots& \Map{n_T-1}{n_T}& \I
\end{pmatrix*}\!.
\end{aligned}\label{lemma:L}
\end{align}
%%%%%%%%%%%%%%%%%%%%%%%%%%%%%%%%%
%%%%%%%%%%%%%%%%%%%%%%%%%%%%%%%%%
 This relation can be verified by direct computation using \eqref{eq:EvolProp}. 

From the definitions, it is evident that applying $\L$ to a vector is preferable over its inverse, and the action of $\L^{-1}$ can be done recursively. Additionally,  $\det \L = 1$ is a useful property for later purposes. For further discussion, we adopt the following notation: $\O := \text{blkdiag}(\O_0,\O_1, \dots, \O_{n_T}) $ and $N_m = n_s\cdot(n_T+1)$.
\subsubsection*{Forecast prior} 
The forecast prior is a distribution obtained by propagating the initial background distribution through the discrete dynamical system. Specifically, the initial state,
$\bmu_0$, is assigned the background distribution
$ \mathcal{N}(\bmu_0^b,\Gb)$,
as described in~
 \Cref{assumption:background_distribution}.
  The subsequent states are then generated by pushing forward 
  $\bmu_0$
 according to the model~\eqref{eq:WC_system}.
Under \Cref{assumption:model_measurement_error},  
$\bm p$ has the distribution $\bm p \sim \mathcal{N}( \bm\mu_{\rm mod},\bm\Gamma_{\rm mod})$ where
 \begin{align}
 \bm\mu_{\rm mod} : = \pmat{ \bmu_0^b  \\ \bm 0}, \qquad \bm\Gamma_{\rm mod} := \begin{pmatrix}
 \Gb &   \bm 0 \\[0.1cm]
%\hline
  \bm 0 &   \Gq
\end{pmatrix}.\label{def:G_mod}
\end{align}
 
 Therefore, from the relation $\bm u = \L^{-1} \bm p$, it follows the forecast prior has the form $\bm u\sim \mathcal{N}\left( \uprior,\Gpr
\right)$,
%%%%%%%%%%%%%%%%%%%%%%%%%%%%%%%%%%%%%%%%%%%%%%%%%%%%%%%%%%%%%%%
where
\begin{equation}
\uprior  =  
\L^{-1}\bm \mu_{\rm mod}, \quad \mbox{ and } \quad
    \Gpr
    = 
    \L^{-1}
    \bm\Gamma_{\rm mod}
    \L\mt.\label{def:distribution_uprior}
\end{equation}
 We denote the density of the 
 resulting Gaussian measure by $\mu_{\rm pr}$. 
 The corresponding density is denoted by $\pi_{\rm pr}$.

\subsubsection*{Posterior distribution} Finally, we are ready to present the posterior distribution obtained in the WC4D-Var context. Under Assumptions~\ref{assumption:data}, \ref{assumption:background_distribution}, \ref{assumption:forward_model}, and \ref{assumption:model_measurement_error}, a straightforward application of Bayes rule gives the posterior distribution $\bmu| \bmyo$, which is Gaussian, denoted by  $\mu_{\rm post} :=\mathcal{N}\left(\bmu_{post}, \Gpost \right)$, where
\begin{equation}\label{def:distribution_upost}
    \upost = \Gpost \left( \O\t \Gmeas^{-1} \bmyo + \Gpr^{-1} \uprior \right),  \qquad \Gpost^{-1} =    \O\t \Gmeas^{-1} \O + \Gpr^{-1},
\end{equation}
where we have used the notation $\Gpr^{-1}= \L\t \Gmod^{-1}  \L$, cf. \eqref{def:distribution_uprior}.

\subsection{Requisite linear algebra concepts}
Here we outline key linear algebra concepts, focusing on non-singular Hermitian matrices  and the 
log-determinant  function, both essential for defining the sensor placement optimality criterion in the next section. The material that appears in this section can be found in standard references such as~\cite{bhatia1997matrix,horn2013matrix}. 

\subsubsection*{Loewner order}
Partial order between positive semidefinite symmetric matrices. We  denote $\mathbf{A} \preceq \mathbf{B}$, indicating that $\mathbf{B} - \mathbf{A}$ is symmetric and positive semidefinite.
\subsubsection*{Sylvester's determinant theorem}  
Also known as the Weinstein–Aronszajn identity, this theorem states that for any matrices $\mathbf{A}$ of size $m \times n$ and $\mathbf{B}$ of size $n \times m$,     
\begin{align} \label{def:SylvestersDeterminant}
\det(\mathbf{I}_m + \mathbf{AB}) = \det(\mathbf{I}_n + \mathbf{BA}),
\end{align}
where $\mathbf{I}_m$ and $\mathbf{I}_n$ are identity matrices.
This holds since $\mathbf{AB}$ and $\mathbf{BA}$ share the same nonzero eigenvalues. We leverage this property, particularly for  $\mathbf{B} = \mathbf{A}\t$.
\subsubsection*{Singular value decomposition} 
The singular value decomposition (SVD) of a matrix \(\A \in \R^{m \times n}\), with \(r \leq  \mathrm{rank}(\A)\), can be expressed as:
\begin{align}\label{def:SVD}
    \A = 
    \begin{bmatrix}
        \U_r & \U_\perp
    \end{bmatrix}
    \begin{bmatrix}
        \bm{\Sigma}_r & \\[0.1cm] & \bm{\Sigma}_\perp
    \end{bmatrix}
    \begin{bmatrix}
        \V\t_r  \\[0.1cm]  \V\t_\perp
    \end{bmatrix},
\end{align}
where   $\bm \Sigma_r =\mathrm{diag}(\sigma_1,\ldots,\sigma_r)$ contains   the \(r\) largest singular values of \(\A\), and $\U_r \in \R^{m \times r}$ and $\V_r \in \R^{n \times r}$ are the matrices whose columns are  left and right singular vectors associated with the singular values in \(\bm{\Sigma}_r\), respectively.

The truncated SVD of $\A$  with \(r \leq \mathrm{rank}(\A)\) provides a  \(\mathrm{rank}\!-r\) approximation of \(\A\) given by  
\begin{align}\label{def:TSVD}
\A\approx \U_r \bm{\Sigma}_r \V\t_r.
\end{align}
\subsubsection*{Kronecker products}   
The Kronecker product between an $m\times n$ matrix $\mathbf{B} = (b_{ij})$ and a $p\times q$ matrix $\mathbf{S}$, denoted $\mathbf{B} \otimes \mathbf{S}$, is a block matrix of size $mp \times nq$:
\[
\mathbf{B} \otimes \mathbf{S} = \begin{pmatrix}
b_{11}\mathbf{S} & \cdots & b_{1n}\mathbf{S} \\
\vdots & \ddots & \vdots \\
b_{m1}\mathbf{S} & \cdots & b_{mn}\mathbf{S}
\end{pmatrix}.
\]

\subsubsection*{Matrix functions}
Let $\mathbf{E}$ be a $d \times d$ real  symmetric matrix with eigendecomposition $\mathbf{E} = \mathbf{U} \Lambda \mathbf{U}\t$, where $\Lambda = \mathrm{diag}(\lambda_1, \ldots, \lambda_d)$, and let $f: \sigma(\mathbf{E}) \to \R$ be a continuous function. The matrix function \( f(\mathbf{E}) \) is then defined as:  
\[
f(\E)= \U \mathrm{diag}(f(\lambda_1), \ldots, f(\lambda_d))\U\t.
\]
In particular, the trace of \( f(\mathbf{E}) \) is given by: 
\begin{align}
\tr f(\E) := \sum_{j=1}^d f(\lambda_j).\label{def:trace_f(E)}
\end{align}
For a positive definite matrix $\E$, we write $\logdet(\E) = \log[\det(\E)]$. Next, we present two useful results about the log-determinant function.
\begin{proposition}
Let \(\E\) be a non-singular symmetric matrix of size \(d \times d\). 
Then, the following identity holds:
\[
\log(|\det(\E)|) = \mathrm{tr}(\widetilde{\log}(\E)),
\]
where \(\widetilde{\log}(x) := \log(|x|)\).
\end{proposition}

\begin{proof}
Since \(\E\) is symmetric, it is diagonalizable by a unitary matrix.
Furthermore, since \(0 \notin \sigma(\E)\), \(\widetilde{\log}\) is continuous in the spectrum of \(\E\). 
The proof then follows from the identity
$
\logdet ~\B =\tr(\log(\B)), 
$
applied to the positive definite  matrix 
$\B = ({\E\t\E})^{\frac 12}=(\E^2)^{\frac 12}=|\E|$.
\end{proof}

\begin{lemma}\label{lemma:logdet_MN}
Let \( \bm M \) and \( \bm N \) be positive semidefinite Hermitian matrices in \( \mathbb{C}^{n \times n} \), and suppose \( \bm M \geq \bm N \) in the Loewner order. Then, the following inequality holds:
\begin{align}
   0\leq  \logdet(\I + \bm M) - \logdet(\I + \bm N) \leq \logdet(\I + \bm M - \bm N).
\end{align}
\end{lemma}

\begin{proof}
 This result can be found in \cite[Lemma 9]{D-Optimal_2018}.
 \end{proof}

% 3. EIG Criterion

\section{Expected information gain and the D-optimal criterion}\label{sec:eig}
%%%%%%%%%%%%%%%%%%%%%%%%%%%%%%%%%%%%%%%%%%%%%%%%%%%%%%%%%%%%%%%%%%%
In this section, we derive a new criterion for OED in the WC4D-Var setting. 
The criterion is based on the expected information gain (EIG) and is proposed in \Cref{ssec:eig4dvar}. 
In \Cref{section:Alternative_formulations}, we propose  alternative formulations of the criterion and compare the computational costs of each approach. We discuss its connections to the corresponding criterion for SC4D-Var in \Cref{ssec:comparison}.

\subsection{Expected information gain for WC4D-Var}\label{ssec:eig4dvar}

In this section, we consider an OED  criterion for the WC4D-Var formulation, based on the concept of EIG. 
More specifically,  the criterion we propose is defined by taking the expectation of the Kullback-Liebler (KL) divergence from  the forecast prior to the posterior distribution, averaged over all possible experimental data.
This is closely related to the D-optimality criterion \cite{Alen, Huan_Jagalur_Marzouk_2024, Kiefer1959OptimumED}. More precisely, we define:  
\begin{align}\label{eqn:eig}
\begin{aligned}
 \EIG(\mu_{\text{post}} \| \mu_{\text{pr}})&:=
  \bm E_{\mu_{\text{pr}}} \left\{   \bm E_{\bm y|\bm u} \left\{  \DKL (\mu_{\text{post}} \| \mu_{\text{pr}})\right\} \right\},
 \end{aligned}
\end{align}
where  $\DKL(\mu_1 \| \mu_2)$ represents the KL divergence 
of $\mu_1$ from $\mu_2$ (with
 corresponding densities  $\pi_1(\bmu)$ and $\pi_2(\bmu)$). If the two measures are  absolutely continuous with respect to a common reference measure, here the Lebesgue measure, the KL divergence can be expressed as 
\begin{equation}\label{eqn:kldef}
D_{\text{KL}}( \mu_1 || \mu_2) = \int \log\left(\frac{d\mu_1(\bmu)}{d\mu_2(\bmu)}\right) \, d\mu_1(\bmu) 
= 
\int \log\left(\frac{\pi_{1}(\bmu)}{\pi_{2}(\bmu)}\right) \,\pi_{1}(\bmu)  d{\bmu}.
\end{equation}

It should be noted that the expression~\eqref{eqn:eig} applies to non-Gaussian prior and posterior distributions. However, under Assumptions~\ref{assumption:data}-\ref{assumption:model_measurement_error},
 we obtain Gaussian prior and posterior distributions for the state, and therefore, we can derive a closed form expression that is the starting point in our numerical investigation.

\begin{proposition}\label{prop:D-opt}
Under Assumptions~\ref{assumption:data}-\ref{assumption:model_measurement_error},  
the EIG for the WC4D-Var method takes the form: 
\begin{align}
\begin{aligned}
\overline{\DKL}(\mu_{\text{post}} \| \mu_{\text{pr}})
&=-\frac{1}{2}\log\det(\Gpost \Gpr^{-1}).%\\ 
\end{aligned}\label{eq:D_optimal}
\end{align}
\end{proposition}
\begin{proof}
The proof of the first equality is similar to that of~\cite[Section 4]{Alen}  and is therefore omitted.
\end{proof}
Although a closed-form expression exists for the EIG, it is not straightforward to compute.
For instance,  the matrices $\Gpost$ and $\Gpr$ cannot be formed explicitly, so $\EIG$ cannot be computed in a straightforward manner. Consequently, we must rely on matrix-free methods to evaluate the objective function, which we discuss in \Cref{sec:preclanc}. 
Dropping the factor of $1/2$, we also define the OED criterion
\begin{equation}
    \label{eqn:criterion}
    \criteria := 2\EIG = \logdet(\Gpr\Gpost^{-1}).
\end{equation}

%%%%%%%%%%%%%%%%%%%%%%%%%%%%%%%%%%%%%%%%%%%%%%%%%%%%%%%%%%%
\subsection{Alternative formulations for EIG}\label{section:Alternative_formulations}

%%%%%%%%%%%%%%%%%%%%%%%%%%%%%%%%%%%%%%%%%%%%%%%%%%%%%%%%%%%
In this section, we present alternative formulations of the WC-4DVAR criterion, and highlight the advantages and disadvantages of these formulations.
The formulations that we propose are based on the relationships between the Schur complement and determinants. 

\subsubsection{Preconditioned formulation}\label{ssec:stronglike}
The first reformulation that we propose resembles the strong constraint formulation, cf. \eqref{eq:DOpt_SC} and \cite[Section 3]{Alen}.  Using the definition of $\Gpost$, cf. \eqref{def:distribution_upost}, and the cyclic property of the determinant, we have:
\begin{align}
\begin{aligned}
\criteria
& = \logdet(\Gpr\Gpost^{-1})\\
&=-\logdet \left( (\Gmod^{-\frac 12}\L)\left( \O\t \Gmeas^{-1} \O +  \L\t \Gmod^{-1} \L \right)^{-1}
                  \left(\L\t \Gmod^{-\frac 12}  \right) \right)\\
&=\> \logdet  \left( \I  + 
\Gmod^{\frac{1}{2}} \L\mt \O\t \Gmeas^{-1}\O
\L^{-1} \Gmod^{\frac{1}{2}}   \right).
\end{aligned}\label{eq:Stable_D-optimality}
\end{align}
We refer to the final expression as the preconditioned formulation, since we can view the posterior covariance preconditioned by the prior covariance matrix. Note that in numerical experiments, we work with a factorization $\Gmod = \G\G\t$ rather than the symmetric square root.
This formulation also bears some resemblance with the EIG obtained in the strong-constraint case, and which is explained in greater depth in \Cref{ssec:sc4d-var}. 

We may simplify the expression for $\criteria$ by defining a new matrix 
\begin{equation}\label{def:A}
\A := \Gmod^{\frac{1}{2}}\L\mt \O\t  \Gmeas^{-\frac 12},
\end{equation}
so that now, $\criteria =\logdet(\I + \A\A\t) = \logdet(\I + \A\t\A)$, by Sylvester's determinant identity, \eqref{def:SylvestersDeterminant}.

In contrast to the preconditioned formulation, we can also consider the unpreconditioned formulation, which takes the form 
\begin{equation}\label{eqn:unprec}
    \criteria =  \logdet(\Gpost^{-1})   + C_U  = \logdet(\O\t \Gmeas^{-1} \O +  \L\t \Gmod^{-1} \L)  + C_U , 
\end{equation}
where the constant $C_U = \logdet(\Gpr)$ is independent of the data and so it is unimportant in the context of OED.

\subsubsection{Saddle-Point formulation-I (SP-I)}
    This approach is referred to as the Saddle formulation for WC4D-Var, see \cite[Section 2]{Gratton}. Consider the matrix:
\begin{align}
\begin{aligned}
\Wb_I & :=
\begin{bmatrix}
  \Gmod & \bm 0 & \L \\
  \bm 0 & \Gmeas & \O \\
  \L\t & \O\t & \bm 0 \\
\end{bmatrix}
=
\begin{bmatrix}
  \I  & \bm 0 \\
 \B &  \I \\
\end{bmatrix}
\left[ 
\begin{array}{c@{}c@{}c}
\Wb 
  & ~\bm 0  \\
  \bm 0  & -\S
\end{array}\right] 
\begin{bmatrix}
  \I  & \B\t \\
 \bm 0 &  \I \\
\end{bmatrix},
 \end{aligned} \label{eq:Saddle1}
\end{align}
where $\S = \O\t \Gmeas^{-1} \O + \L\t \Gmod^{-1} \L$, $\B =[\L\t~\O\t ]\Wb^{-1} $, and  $\Wb=
    \begin{bmatrix}
  \Gmod & \bm 0  \\
  \bm 0 & \Gmeas 
\end{bmatrix}.$ \newline
This matrix  is related to the Hessian of $\mathcal{J}_{\text{WC}}$, cf.  \eqref{def:Jwc}; for further details, see \cite{boyd2004convex, HIGHAM1998261}. 
From   \eqref{eq:Saddle1}, we have $|\det \Wb_I| = \det(\W)\det(\S)$, and therefore
\begin{align*}
    |\det \Wb_I| 
            =&  \det\left(\Gmod \right)\det\left(\Gmeas \right) \det\left( 
             \O\t \Gmeas^{-1} \O + \L\t \Gmod^{-1} \L \right).
\end{align*}
 Then  $\criteria = \log|\det \Wb_I| + C_{I}$, where the constant term $C_{I} = \logdet(\Gmeas)$, can be disregarded in the context of OED since it is independent of the data. The main advantage of this approach is that a matrix-vector product (henceforth, matvec) with $\W_I$ circumvents the need for $\L^{-1}$ and its transpose.
 Similarly, this matrix does not involve the inverse of $\bm\Gamma_{\bm R}$ and square root of $\Gmod$. 

\subsubsection{Saddle-Point formulation-II (SP-II)} 
The second method focused on the linear system is related to Lagrange Multipliers with a penalization term, see  \cite{Benzi}  and \cite[Section 2]{Gatica}. 
Let us consider the matrix
 \begin{align}
  \Wb_{II} =
\begin{bmatrix}
  \Gmod  & \L \\
 \L\t &  -\O\t \Gmeas^{-1} \O \\
\end{bmatrix}= 
\begin{bmatrix}
  \I  & \bm 0 \\
 \L\t \Gmod^{-1} &  \I \\
\end{bmatrix}
\begin{bmatrix}
  \Gmod  & \bm 0 \\
 \bm 0 &  \S \\
\end{bmatrix}
\begin{bmatrix}
  \I  & \Gmod^{-1} \L \\
  \bm 0 &  \I \\
\end{bmatrix},\label{eq:Saddle2}
\end{align}
where $\S= -(\O\t \Gmeas^{-1} \O +\L\t \Gmod^{-1} \L )$ is the Schur complement of $\Gmod$.  
The absolute value of the determinant of  $\Wb_{II}$ satisfies:
\begin{align*}
    |\det \Wb_{II}| 
     &=\> \det\left( \Gmod \right) \det\left( \O\t \Gmeas^{-1}\O + \L\t \Gmod^{-1} \L \right).
\end{align*}
Thus,  $\criteria = \log|\det \Wb_{II}|$, since $\det \L^{-1} =1.$ As with the other saddle point formulation, $\W_{II}$ avoids the inverses of $\L$ and its transpose. 

\subsubsection{Comparison of computational costs}
The above discussion suggests various ways of computing the
EIG, using either the matrix $\A\t\A$, $\W_{I}$, or $\W_{II}$. 
 The cost of a single matvec with these matrices and the spectral properties of the involved matrices is central to matrix-free methods for evaluating the log-determinant. 
The required matvecs in each case are summarized in \Cref{tab:formulations_comparison}, and are further discussed in \Cref{sec:preclanc}.

Let  \(\cost{X}\) represent the computational cost of a matvec with the matrix \(\bm{X}\),  measured in terms of the number of flop operations required.
\Cref{tab:formulations_comparison} outlines the matrix sizes and highlights whether square roots or inverse matrices are involved. 
It is worth noting that $\cost{\L}$ and $ \cost{\L^{-1}} $ are both approximately $(n_T+1)\cost{M}$, assuming the cost of applying the matrix 
$ \Map{\ell}{\ell+1}$ for $0\leq \ell \leq n_T-1$ is constant, denoted by $\cost{M}$. 
Although, $\cost{\L^{-1}}$  requires additional memory access. 
Finally, when  
$\Gmeas$ is diagonal, then  $\cost{\Gmeas}=\cost{\Gmeas^{-1}} = \mathcal{O}(N_m).$

\begin{table}[ht!]
    \centering
    \renewcommand{\arraystretch}{1.14} 
    \begin{tabular}{c|c|c|c|c|c} 
        \textbf{Formulation/Matrix} & $\Gmeas^{-1}$    & $\Gmod^{-1}$ & $\Gmod^{\frac 12}$ & $\L^{-1} / \L\mt$  & \textbf{Matrix Size} \\ \hline
        \textbf{Unpreconditioned} \cref{eqn:unprec}     & \checkmark  & \checkmark & \xmark& \xmark  & $N_d$\\ 
        \textbf{Preconditioned} \eqref{eqn:criterion}   & \checkmark  & \xmark & \checkmark & \checkmark & $N_d$\\
        \textbf{SP-I}  \eqref{eq:Saddle1}     & \xmark & \xmark &  \xmark &\xmark & $N_d+N_m$\\ 
        \textbf{SP-II}  \eqref{eq:Saddle2}   & \checkmark & \xmark & \xmark &  \xmark  & $2N_d$
    \end{tabular}
    \caption{A summary of matrices involved for all the different formulations for computing the WC4D-Var criterion $\criteria$.} 
    \label{tab:formulations_comparison}
\end{table}

%%%%%%%%%%%%%%%%%

\subsection{Experimental design}\label{ssec:exptdesign} 
In data assimilation, data is typically collected in the form of sensor measurements. This is the scenario we consider in this work. 
We assume that there are $n_s$ candidate sensor locations and there is a budget of $k$ sensors, where  $1 \le k \le n_s$. We must find the optimal sensor locations, which means selecting $k$ sensors out of $n_s$. 
We consider the case where sensor locations remain fixed over time, namely $\O := \text{blkdiag}(\O_0, \dots, \O_0) $.

Assuming that the observation noise is diagonal, 
 each column of  $\A$ in  \eqref{def:A},  corresponds to a specific sensor and snapshot. 
To formalize this, we make the following additional assumption:
\begin{assumption}\label{assumption:diag_R} 
We assume that $\bm{R}_\ell = \sigma_\ell^2 \I$ for $0 \le \ell \le n_T$,  meaning that the spatial measurement noise is uncorrelated and has a constant variance.
\end{assumption}
This assumption is made for notational simplicity, but the approaches we describe can handle diagonal noise covariance matrices.

Then, the  matrix \(\A\) can be partitioned into a block of columns, one for each snapshot, as follows:
\begin{align} \label{eq:partitionA}
\A = \bmat{\A_0 & \cdots & \A_{n_T}}\in \R^{N_d \times N_m}, \quad  \A_i \in \R^{N_d\times n_s}.
\end{align}
To encode the selection of \(k\) sensors out of \(n_s\), we introduce a selection matrix \(\S \in \R^{n_s \times k}\), consisting of \(k\) independent columns of the \(n_s \times n_s\) identity matrix. This selection is given by:
\[
\A (\I \otimes \S) = \bmat{\A_0 \S & \cdots & \A_{n_T} \S}\in \R^{N_d \times ((n_T+1)k)}.
\]
Then, the EIG associated with those $k$ sensors is given by:
\begin{align}\label{def:EIG_S}
\criteria(\S) := \logdet( \I + \A (\I \otimes \S) (\I \otimes \S\t)\A\t).
\end{align}

With the selection operator $\S$, the data collection takes the form
\[ \S\t\bmyo_{\ell} = \S\t\O_\ell\bmu_\ell + \S\t\bm{r}_{\ell}, \qquad 0 \le \ell \le n_T.\] 
The OED problem can then be expressed as:

\[ \max_{\S} \criteria(\S) = \logdet( \I + \A (\I \otimes \S) (\I \otimes \S\t)\A\t),\]
where the optimization is performed over $k$ independent columns from the $n_s\times n_s$ identity matrix. 
It is well known that $\criteria(\S)$ is a nondecreasing submodular function of the selected sensor set, see \cite{maio2025submodularity} for a constructive proof. Consequently, the standard greedy algorithm provides a $(1 - e^{-1})$–approximation to the global optimum under a cardinality constraint~\cite{nemhauser}. In \Cref{sec:sensor_placement}, however, we introduce a novel heuristic that, in our numerical experiments, almost always attains equal or higher EIG values than the greedy baseline.

\subsection{Comparison with SC4D-Var}\label{ssec:comparison}
Analogous to the WC4D-Var, we can define the EIG for the SC4D-Var as the expected KL divergence from the background to the strong-constraint posterior distribution. This takes the form  
\begin{align}\label{eqn:eigdef}
\begin{aligned}
 \EIG(\mu_{\text{post}}^{\rm sc} \| \mu_{\text{back}})&:=
  \bm E_{\mu_{\text{pr}}} \left\{   \bm E_{\bm y|\bm u} \left\{  \DKL (\mu_{\text{post}}^{\rm sc} \| \mu_{\text{back}})\right\} \right\},
 \end{aligned}
\end{align}
Following the technique in \cite[Section 4]{Alen}, we can obtain an explicit formula for the EIG in the case  for the SC4D-Var framework
\begin{align}
\EIG[\mu_{\rm post}^{\rm sc} || \mu_{\rm back}]=& \> \frac{1}{2}  
\logdet\left(  \Gb^{\frac 12} \left( {\Gb^{-1}}  
+ \sum_{\ell = 0}^{n_T}
  (\O_\ell \Map{0}{{\ell}})\t 
  \bm{R_\ell}^{-1}
  (\O_\ell \Map{0}{{\ell}})\right) {\Gb^{\frac 12}} \right) \nonumber
  \\  = & \frac{1}{2}\logdet\left(  \I +     
\sum_{\ell = 0}^{n_T}
  (\O_\ell \Map{0}{\ell} \Gb^{\frac 12})\t 
  \bm{R}_\ell^{-1}
  (\O_\ell \Map{0}{\ell}\Gb^{\frac 12})\right)  . \label{eq:DOpt_SC}
\end{align}
A key distinction in the weak-constraint case is that both the forecast prior and posterior distributions are defined over the entire discrete trajectory of the states, as the weak-constraint accounts for model error propagation throughout the trajectory. In contrast, the strong-constraint approach only estimates the initial condition, since the trajectory is determined by the dynamical system.

Similar to WC4D-Var,  the OED criterion  
can be expressed as follows:
\begin{align}
\begin{aligned}
   \criteria^\text{SC} := & \> \logdet\left(  \I +     
\sum_{\ell = 0}^{n_T}
  (\O_\ell \Map{0}{\ell} \Gb^{\frac 12})\t 
  \bm{R}_\ell^{-1}
  (\O_\ell \Map{0}{\ell}\Gb^{\frac 12})\right).
\end{aligned}\label{eq:EIG_SC}
\end{align}
If we incorporate the selection matrix into the EIG criteria for the SC4D-Var we have:
\[ \criteria^\text{SC} (\S) := 
 \logdet\left(  \I +     
\sum_{\ell = 0}^{n_T}
  \left(\S\t\bm{R}_\ell^{-\frac 12} \O_\ell \Map{0}{\ell} \Gb^{\frac 12}\right)^\top 
\left(\S\bm{R}_\ell^{-\frac 12} \O_\ell \Map{0}{\ell} \Gb^{\frac 12}\right)\right).
\] 
In the following proposition, we derive bounds for the difference in the EIG criteria for WC4D-Var and SC4D-Var. 
\begin{proposition}\label{proposition:bound_SCandWCoptCriteria}
The following holds: 
\begin{align} \label{eq:gapEIG_SC_WC}
    0\leq \criteria  -   \criteria^\text{SC} 
      \leq \logdet \left( \I+ \bm Z
\begin{bmatrix}
 \bm 0 &    \\
   &   \Gq
\end{bmatrix}
\bm Z\t\right),
\end{align}
where $\Z := \Gmeas^{-\frac 12} \O \L^{-1}$.
Furthermore, as $\Gq \to \bm{0}$, the WC4D-Var EIG criterion converges to the SC4D-Var EIG criterion, i.e., $\criteria \to \criteria^\text{SC}$.

\end{proposition}
\begin{proof}
We first note that by the Sylvester determinant identity,  $\criteria$ in \eqref{eq:Stable_D-optimality} can be expressed as
\[ \criteria = \logdet(\I + \Z\Gmod\Z\t).\]
Next, a straightforward calculation shows that
\begin{align*}
     \bmat{ \B^{\frac 12} \\ & \bm{0} }\Z\t \Z  \bmat{ \B^{\frac 12} \\ & \bm{0} }    =
    \bmat{
 \displaystyle  \sum_{\ell=1}^{n_T} 
\left(\O_\ell\Map{0}{\ell} \Gb^{\frac 12}\right)\t
 {\bm R}^{-1}_\ell
\O_\ell \Map{0}{\ell} \Gb^{\frac 12}   &   \boldsymbol{0} \\[0.1cm]
%\hline
  \boldsymbol{0}  &   \boldsymbol{0}
}.
\end{align*}
Once again,  using the Sylvester determinant identity, we have
\[ \criteria^\text{SC} = \logdet\left(\I + \Z \bmat{ \B \\ & \bm{0} }\Z\t\right).\] 
Consider $\Z\Gmod\Z\t$, which can be expressed as a sum
\[ \Z\Gmod\Z\t = \Z \bmat{ \B \\ & \bm{0} }\Z\t + \Z \bmat{ \bm{0}\\ & \Gpr  }\Z\t. \]
Then, \eqref{eq:gapEIG_SC_WC} follows from \Cref{lemma:logdet_MN} with $\M = \Z\Gmod\Z\t$ and $\N = \Z \bmat{ \B \\ & \bm{0} }\Z\t$. 
Finally, the convergence $\criteria \to \criteria^\text{SC}$ as $\Gq$ vanishes follows directly from \eqref{eq:gapEIG_SC_WC} and the continuity of the determinant function.
\end{proof}

This result shows that if the model error is large in the sense of the upper bound in \Cref{proposition:bound_SCandWCoptCriteria}, then the difference between the two EIG criteria can be large. To put it differently, ignoring the model error in SC4D-Var can lead to vastly different criteria for the same design.

\subsubsection*{Optimal sensor placement}
\Cref{proposition:bound_SCandWCoptCriteria} 
assumes data is collected at all the sensors. We now derive an analogous bound between the two criteria for the optimal designs from the two criteria. 
\begin{proposition}\label{proposition:SCandWCoptCriteriaGap}

Let \( \S^*_\text{WC} \) and \( \S^*_\text{SC} \) denote  optimal selection matrices for the WC4D-Var and SC4D-Var EIG criteria, namely:
\begin{align}
    \S^*_\text{WC} &\in \argmax_{\S} \criteria(\S), \quad
    \S^*_\text{SC} \in \argmax_{\S} \criteria^\text{SC}(\S).
\end{align}
Then, the gap between the criteria values for these optimal selections is bounded as:
\begin{align*}
    0\leq \criteria(\S^*_\text{WC}) - \criteria(\S^*_\text{SC}) \leq \logdet \left( \I + (\I \otimes \S^*_\text{WC})\bm{Z} 
    \begin{bmatrix}
    \bm{0} & \\
    & \Gq
    \end{bmatrix} 
    \bm{Z}^\top (\I \otimes \S^*_\text{WC})^\top \right).
\end{align*}

\end{proposition}

\begin{proof}
First, let \( \S \) be any selection matrix corresponding to the choice of \( k \) elements  out of \( n_s \). 
Then, the difference between the EIG criteria for the WC4D-Var and SC4D-Var frameworks satisfies the following bound:
\begin{align}\label{eq:gapEIG_SC_WC2}
    0 \leq \criteria(\S) - \criteria^\text{SC}(\S) \leq \logdet \left( \I + (\I \otimes \S)\bm{Z} 
    \begin{bmatrix}
    \bm{0} & \\
    & \Gq
    \end{bmatrix} 
    \bm{Z}^\top (\I \otimes \S)^\top \right),
\end{align}
where \( \bm{Z} := \Gmeas^{-\frac{1}{2}} \O \L^{-1} \).
The inequality \eqref{eq:gapEIG_SC_WC2} follows a similar approach to the proof of  \Cref{proposition:bound_SCandWCoptCriteria}. 
Now we turn to the actual result we wish to prove. The lower bound follows from the optimality of $\S^*_\text{WC}$.  For the upper bound, we decompose 
\[\begin{aligned}
    \criteria(\S^*_\text{WC}) - \criteria(\S^*_\text{SC}) = & \>  \underbrace{\criteria(\S^*_\text{WC}) - \criteria^\text{SC}(\S^*_\text{WC})}_{\equiv \alpha}
     + \underbrace{\criteria^\text{SC}(\S^*_\text{WC}) -  \criteria^\text{SC}(\S^*_\text{SC})}_{\equiv \beta}.
\end{aligned} \]
We have $\beta \le 0$ by  the optimality of $\S^*_\text{SC}$. The upper bound then follows from~\eqref{eq:gapEIG_SC_WC2} applied to $\S^*_\text{WC}$.
The uniform upper bound follows directly from Sylvester's law of inertia. 
\end{proof}
We also have the simpler but looser bound  
$$ 0 \le \criteria(\S) - \criteria^\text{SC}(\S) \leq \displaystyle \logdet \left( \I+ \bm Z
\begin{bmatrix}
 \bm 0 &    \\
%\hline
   &   \Gq
\end{bmatrix}
\bm Z\t\right),$$
where the upper bound is independent of $\S^*_\text{WC}$. 

Once again, we see that the gap between the two criteria evaluated at the optimal sensor placement (computed for each criterion)
can be large, and the bound may become loose or noninformative if the model error is large.

% 4. Evaluating criterion

%%%%%%%%%%%%%%%%%%%%%%%%%%%%%%%%%%%%%%%%%%%%%%%%%%%%%%%%%%%%%%%
%%%%%%%%%%%%%%%%%%%%%%%%%%%%%%%%%%%%%%%%%%%%%%%%%%%%%%%%%%%%%%%
\section{Evaluating the EIG  criterion for  WC4D-Var}\label{sec:preclanc}
%%%%%%%%%%%%%%%%%%%%%%%%%%%%%%%%%%%%%%%%%%%%%%%%%%%%%%%%%%%%%%%%
%%%%%%%%%%%%%%%%%%%%%%%%%%%%%%%%%%%%%%%%%%%%%%%%%%%%%%%%%%%%%%%%
In this section, we develop matrix-free methods to evaluate the objective function $\criteria$. We assume that the selection operator $\S$ that determines the selected sensors, takes the form $\S = \I$ to save on notational complexity, but all the methods extend easily to the case $\S \neq \I$. 

In \Cref{sec:eig}, we presented several different formulations of the EIG criterion. We now show that each formulation can be written (up to a possibly additive constant) in the form $\tr f(\E)$, for a suitable function $f$ and a nonsingular symmetric  matrix $\E \in \R^{d\times d}$.

\begin{enumerate}
    \item Preconditioned formulation: From the relation, 
    \[ \logdet(\I + \A\A\t) = \tr\log(\I + \A\A\t),\]
    we can see that $\criteria = \tr f(\E)$ for $f(x) = \log(x)$ and $\E = \I+\A\A\t$. 
    \item Unpreconditioned formulation: From  \eqref{eqn:unprec} it follows that  
    \[\criteria = \tr f(\E) +C_U,\] where  
    $\E = \Gpost^{-1}, f(x)= \log(x)$, and $C_U = \logdet(\Gpr)$.  
    \item SP-I formulation: From the relation
    \[ \log|\det\W_I| = \tr\log| \W_I|, \]
 we can see that $\criteria = \tr f(\E) + C_I$ for $f(x) = \log(|x|)$,  $\E = \W_I$, and $C_{I} = \logdet(\Gmeas)$. From~\eqref{eq:Saddle1} and Sylvester's inertia theorem~\cite[Theorem 4.5.8]{horn2013matrix}, we can see that $\W_I$ has $2N_d$ positive and $N_m$ negative eigenvalues and is invertible.
    \item SP-II formulation: Similarly to the SP-I formulation,  from  \eqref{eq:Saddle2}, we have:  $\criteria = \tr f(\E)$, where  $\E = \W_{II}$ and $f(x) = \log(|x|)$. From~\eqref{eq:Saddle2}, $\W_{II}$ has $N_d$ positive and $N_d$ negative eigenvalues, and so it is invertible.
\end{enumerate}

We can now discuss efficient methods for the computation of $\criteria$, using trace estimators. Note that, we do not discuss the computation of $C_U$ and $C_I$.

\subsubsection*{Stochastic Lanczos Quadrature} Based on \cite{Bai_Golub,Ubaru}, we consider this method that combines the Lanczos method with Hutchinson's trace estimator \cite{Hutchinson}. The Hutchinson's estimator approximates the trace of $f(\E)$ as follows:
 \begin{align}\label{def:Hutchinsons_estimator}
     \tr (f(\E))\approx \frac{1}{N}\sum_{\ell =1 }^N \bm z\t_\ell f(\E)\bm z_\ell,
 \end{align}
where $\bm z_\ell$ are independent Rademacher random vectors (with independent entries $\pm 1$ with equal probability). 
 
Employing Hutchinson's estimator to $f(\E)$ only  requires matvecs of the form $f(\E) \bm z_\ell$.  
 The Lanczos method is an iterative algorithm that can approximate $f(\E) \bm z_\ell$ only via matvecs involving  $\E$, making it suitable for large-scale matrices. 
 In the Lanczos process, given the starting vector $\bm{v}_1^\ell:=\bm{z}_\ell/\|\bm{z}_\ell\|_2$, we run $n_\text{iter}$ steps of the Lanczos iteration to obtain the decomposition $\E\V_{n_\text{iter}}^\ell = \V_{n_\text{iter}}^\ell\T^\ell_{n_\text{iter}} + \beta_{n_\text{iter}+1}^\ell\bm{v}_{n_\text{iter}+1}\bm{e}\t_{n_\text{iter}}$. 
 Here $\V_{n_\text{iter}} ^\ell = \begin{bmatrix} \bm{v}_1^\ell & \dots &\bm{v}_{n_\text{iter}}^\ell\end{bmatrix}$ 
 has orthonormal columns (in exact arithmetic) and $\T^\ell_{n_\text{iter}}$ is a tridiagonal matrix, whose eigenvalues approximate some of the eigenvalues of $\E$.

 The stochastic Lanczos quadrature (SLQ) method was designed to approximate $\tr( f(\E))$, by combining the Hutchinson trace estimator along with the Lanczos method.  In particular, the SLQ method uses the estimate
 \begin{align}
     \tr (f(\E))\approx \frac{d}{N}\sum_{\ell =1 }^N \bm e\t_1\! f\!\left(\T^\ell_{n_\text{iter}}\right)\bm e_1,
 \end{align}
where the factor 
$d$ comes from $\|\bm z_\ell\|_2^2 =d.$  
In our implementation, we used Lanczos with full reorthogonalization.
The number of iterations, \( n_\text{iter} \), is the count needed for the relative change in \( \tr\!\left(f(\T^\ell_{n_\text{iter}})\right) \) between consecutive iterations to fall below the predefined tolerance, set here to \( 10^{-10} \).

Since, in general,  the Lanczos method favors the leading  eigenvalues of  $\E$, 
 we expect few Lanczos iterations if the function $f$ also prioritizes these eigenvalues.

As we shall see in the numerical examples, cf. \Cref{tab:SLQ_summary},  the convergence of the SLQ 
method is relatively slow with respect to the number of samples.
To address this, the following approach enhances convergence by extracting additional information from the vectors  $f(\E) \bm z_\ell$ through a randomized Nyström approximation.
This method is particularly advantageous when the action  $f(\E) \bm z_\ell$
is computationally expensive compared to the Nyström method.

\subsubsection*{\xntrace{}+Lanczos}

Recent work by Epperly et al. \cite{Epperly_2024}  has introduced several trace estimators that are more accurate than Hutchinson's estimator. One such method, \xntrace{}, utilizes the Nyström approximation to the matrix $\BPsi$:
\begin{align}
\BPsi\langle \X \rangle := (\BPsi\X) \left( \X\t \BPsi\X \right)^\dagger (\BPsi\X)\t \quad \text{for a test matrix } \X \in \R^{m \times k}.
\end{align}
This method is designed for positive semidefinite matrices. Therefore, in the context of the present work, it only applies to the Preconditioned formulation $\tr(\log(\I + \A\A\t)).$ That is, we take $\BPsi =  \log(\I + \A\A\t)$.

In the  \xntrace{} approach, cf. \cite[Section 2.3]{Epperly_2024}, we draw a test matrix $\bm{\Omega} = [\bm \omega_1 \,\dots\, \bm\omega_N]$, where
$\bm \omega$  follows a  spherically symmetric distribution, 
and then  define the trace estimator as follows:
\begin{align*}
   \widehat{\tr}_{XN}(\BPsi;\bm \Omega) :=\frac{1}{N}\sum_{j=1}^N
   \left(
   \tr (\BPsi \langle \bm \Omega_{-j} \rangle) +\bm \nu\t_j(\BPsi   \bm \nu_j)\right),
\end{align*}
where 
$\bm \Omega_{-j} $ is obtained by removing the $j$-th column of  $\bm \Omega$, leave-one-out approach, and $\bm \nu_j$ is obtained from  $\bm \omega_j$  by removing its projection onto the column space of  $\bm \Omega_{-j}$. For additional details of the method and its implementation, we refer to~\cite[Section 2.2]{Epperly_2024}. 
The action of $\BPsi\langle\X\rangle$ is approximated using the Lanczos method, similar to how it is employed in SLQ. We keep the same setting for the Lanzcos iterations for both methods.

\begin{table}[ht!]
\centering
\begin{tabular}{@{}lll@{}}
\toprule
\textbf{Algorithm}                   & \textbf{Operations (f{}lops)}  & \textbf{Reference}\\ \midrule
\textbf{ SLQ}            & \(N\cdot \cost{f(\E)}=N(n_\text{iter}\cost{\E}+d\cdot n^2_\text{iter} )\)  &  \cite[Section 3]{Ubaru}\\
\textbf{\xntrace{}+Lanczos}               & \(N\cdot \cost{f(\E)}+\mathcal{O}(N^2 d)\)  & \cite[Section 2.2]{Epperly_2024} \\
\bottomrule
\end{tabular}
\caption{Computational complexity of \xntrace{} and Stochastic Lanczos Quadrature algorithms. Here, \(N\) denotes the number of random probes,  and \(\cost{f(\E)}\) represents the computational cost of a single application of \(f(\E)\) via Lanczos, where  \(n_\text{iter}\) is the number of Lanczos iterations.  The cost  \(\cost{\E}\) will depend of each formulation; see  Table~ \ref{tab:formulations_comparison}.}
\label{tab:complexity_comparison}
\end{table}

We summarize the computational costs for the SLQ and \xntrace{} +Lanczos algorithms in Table~\ref{tab:complexity_comparison}. The table omits the costs associated with generating test matrices for clarity and focuses on the dominant terms.

% 5. Subset selection

\section{Subset Selection for Optimal Sensor Placement}\label{sec:sensor_placement}

This section discusses near-optimal experimental design within the WC4D-Var framework using a column subset selection approach. We focus on the Golub-Klema-Stewart (GKS) method, as applied to sensor placement in \cite{eswar2024bayesian}, specifically for measurements collected at the final time step.
Then, we extend this approach to the case that measurements are collected in time at indices $0 \le \ell \le n_T$ and propose a randomized variant. 

\subsubsection*{Golub-Klema-Stewart approach} 
Suppose,  we only collect data at time step $n_T$. 
The key insight in~\cite{eswar2024bayesian} is that selecting sensors is equivalent to selecting columns from $\A$. That is, we need to select the columns $\A_{n_T}\S$. The  approach has two steps. First, we compute the truncated SVD of $\A \approx \U_k \bm\Sigma_k \V\t_k$. Second, we compute a pivoted QR factorization of $\V\t_k$ as 
\[\V\t_k \bmat{\bm\Pi_1 & \bm\Pi_2} = \bm{\Psi}_1 \bmat{\bm{T}_{11} & \bm{T}_{12}},\]
where $\bm{\Psi}_1 \in \R^{k \times k}$ is orthogonal, $\bm\Pi_1 \in \R^{k\times k}$ is a permutation matrix, and $\bm{T}_{11}$ is an upper triangular matrix. Setting the selection matrix $\S = \bm\Pi_1$,
it was shown in~\cite{eswar2024bayesian} that
\begin{equation}\label{eqn:eigbound}
    \logdet(\bm{I} + \bm\Sigma_k^2/q(n_s,k)^2)\le  \logdet(\I + \A_{n_T}\S\S\t\A_{n_T})  \le \logdet(\bm{I} + \bm\Sigma_k^2).
\end{equation}   
Here, $q(n_s,k)$ is a factor that depends on the specific implementation. If strong rank-revealing QR (sRRQR) is used~\cite[Algorithm 4]{gu1996efficient}, with parameter $f = 2$, then $q(n_s,k) = \sqrt{1 + 4k(n_s-k)}$. Thus, the selected columns are nearly optimal, as can be seen from~\eqref{eqn:eigbound}.
In our numerical experiments, we  use column pivoted QR (CPQR) for which~\eqref{eqn:eigbound} $q(n_s,k) = \sqrt{n_s-k}\cdot 2^k$. Although the bound is weaker than for sRRQR, CPQR is much cheaper than sRRQR and has nearly the same performance empirically. This algorithm can be implemented in a matrix-free fashion. Only the truncated SVD requires access to $\A$ and can be computed matrix-free using either a Randomized SVD \cite{Halko} or a Krylov subspace method \cite{Matrix_Computations}.

\subsubsection*{Extension to time-dependent setting} We now discuss how to extend the approach  
in~\cite{eswar2024bayesian}
to the time-dependent case, i.e., data is collected at times $0\le \ell \le n_T$. 
We may partition the matrix $\A$ as in \eqref{eq:partitionA}:
\[ \A = \bmat{\A_0 & \cdots & \A_{n_T}}\in \R^{N_d \times (n_Tn_s)}. \]
The GKS approach can be applied independently to each block, but it results in a different set of columns (sensors) for each time step. However, we need a strategy to ensure that the same columns are selected from each block $\A_i$ for $1 \le i \le n_T$. 

To this end, we consider the reshaped matrix $\AR$ defined as 
\begin{equation}
    \label{eqn:reshaped}
    \AR = \bmat{\A\t_0 & \cdots &\A\t_{n_T}}\t \in \R^{N_d(n_T+1) \times n_s}.
\end{equation}
It is clear that by selecting columns from the reshaped matrix, we are, in effect, selecting the same columns from each block. Therefore, we now apply GKS on this reshaped matrix.
This is described in \cref{alg:GKS_sensorPlacement}. 
 Note that $\AR$ need not be formed explicitly, and the truncated SVD can be implemented in a matrix-free fashion.

We now derive a result on the performance of any subset selection algorithm. 
Let $K=(n_T+1)k$ and let $\bm{A} \approx \U_K\bm\Sigma_K\V\t_K$ represent the truncated SVD. We have the following result:
\begin{proposition}
    Let $\rank(\V\t_K (\S \otimes  \I)) = K$  and define $\zeta := \| (\V\t_K (\S \otimes  \I))^{-1}\|_2$. Then, 
\begin{equation}\label{eqn:eigbound2}
    \logdet(\bm{I} + \bm\Sigma_K^2/\zeta^2) \le  \criteria(\S) \le \criteria(\S^*_{\rm WC}) \le \logdet(\bm{I} + \bm\Sigma_K^2).
\end{equation}   
\end{proposition}
\begin{proof}
    This follows readily from the arguments in~\cite[Theorem 3.2]{eswar2024bayesian}.
\end{proof}

\begin{algorithm}[!ht]
\caption{Sensor placement using column subset selection}
\label{alg:GKS_sensorPlacement}

\begin{algorithmic}[1]
\Require ${\A} \in \mathbb{R}^{N_d \times (n_T \cdot n_s)}$: Data matrix, $n_T$: Number of snapshots, 
 $n_s$: Number of candidate sensors, 
$k$ : Number of active sensors to select 
\State Construct $\AR$ as in~\eqref{eqn:reshaped}
\State Compute the top \(k\) right singular vectors \(\V_k\) from the truncated SVD of \(\AR\).
\State  Perform column-pivoted QR  on \(\V_k^\top\) to obtain the permutation matrix \(\bm{\Pi}\).
%\Comment{$\V_k^\top \bm\Pi = \bm{QR}$}
\State Set \(\S\) as the  first \(k\) columns of \(\bm{\Pi}\). 
\Comment{Select the top-\(k\) sensor indices.}
\State \Return \(\S\): Indices of the $k$ most informative sensors 
\end{algorithmic}
\end{algorithm}

To analyze the computational cost, let  $T_{\AR}$ represent the cost of a matvec with $\AR$ and its transpose. The cost of a truncated SVD is $\mc{O}(kT_{\AR} + k^2N_dn_T)$ flops whereas the CPQR costs $\mc{O}(k^2n_s)$ flops.

\subsubsection*{Randomized approach} In this approach, we draw a random matrix $\bm\Omega \in \R^{D\times N_n}$ and compute the sketched matrix 
\begin{align} \label{eq:Rand_sensor_selection}
\bm{Y} = \bm\Omega \A \in \R^{D\times N_m}. 
\end{align}
The random matrix $\bm\Omega$ has 
independent $\mc{N}(0,1/D)$ entries. The number of rows $D$ is taken based on the formula $D = n_T\cdot k + p$, where $p$ is a small oversampling parameter, e.g., $p \approx 20$. We then instantiate \Cref{alg:GKS_sensorPlacement} with $\bm{Y}$ rather than $\A$ and in step 1, instead of $\AR$, we reshape to get $\Y_{\bm{R}} \in \R^{Dn_T \times n_s}$. We omit a detailed description of the algorithm.  The computational cost of this algorithm is \[(DT_{\A}) + \mc{O}(Dn_T n_s^2 + n_sk^2)  \> \text{flops}.  \] 
Here $T_{\A}$ is the cost of a matvec with $\A\t$.

This randomized approach has several advantages  over \cref{alg:GKS_sensorPlacement}. First, we do not need to construct or work with $\AR$ and instead we can work directly with $\A$. Second, the algorithm does not require adjoints of the forward operator. Observe that we can form $\Y$ as 
\[ \Y\t = \A\t \bm\Omega\t = \Gmeas^{-\frac 12}\O\L^{-1} \Gmod^{\frac 12}\bm\Omega\t. \]
This approach was first proposed in~\cite{eswar2024bayesian} and was referred to as randomized adjoint-free OED (RAF-OED). The key advantage of this method is that, for PDE-based problems, it eliminates the need for adjoint PDE solves. This is important in applications where the adjoint is unavailable (e.g., due to legacy codes), erroneous, or is more expensive than the forward problem. The only difference here is that we extend it to the time-dependent case.

The main intuition behind the randomization is that the matrix $\bm{Y}$ captures essential information regarding the row-space of the matrix. To give some quantitative justification for the approach, let $\bm{C} := \bm{A}(\S \otimes \I) \in \R^{N_d\times (n_Tk)}$ denote the matrix corresponding to any selection of $k$ sensors and let $\bm{C} =\bm{U}_C\bm\Sigma_C\bm{V}\t_C$ denote the thin-SVD of $\bm{C}$. 
By applying the selection operator to the sketch $\widetilde{\bm{Y}} = \bm{Y}(\S \otimes \I) = \bm{\Omega}\bm{A}(\S \otimes \I) = \bm\Omega\bm{C} $, we can show by using singular value inequalities  
\[   \frac{\sigma_j(\bm{C})}{ \| (\bm\Omega \bm{U}_C)^\dagger\|_2}\le \sigma_j(\widetilde{\bm{Y}}) \le \sigma_j(\bm{C}) \| \bm\Omega \bm{U}_C\|_2, \qquad 1 \le j \le K.\]
Note that by orthogonal invariance, $\bm\Omega \bm{U}_C \in \R^{D\times K}$ is a standard Gaussian random matrix. By using~\cite[Section 7.3]{vershynin2018high}, with probability at least $1-2e^{-t^2/2}$,  
\[ \frac{\sigma_j(\bm{C})}{1 - \sqrt{K/D} -t }\le \sigma_j(\widetilde{\bm{Y}}) \le \sigma_j(\bm{C}) \left( 1 + \sqrt{\frac{K}{D}}  + t\right), \qquad 1 \le j \le K .\]
For example, for any $ 0 < \epsilon, \delta < 1$,  we can choose $D =\epsilon^{-2}(\sqrt{K}+ \sqrt{2\ln(2/\delta)})^2$ to get
\[  \frac{\sigma_j(\bm{C})}{1 - \epsilon }\le \sigma_j(\widetilde{\bm{Y}}) \le \sigma_j(\bm{C}) \left( 1 + \epsilon\right), \qquad 1 \le j \le K ,\]
with probability at least $1-\delta$. Therefore, the singular values of $\widetilde{\bm{Y}}$ can be considered to be good approximations to the singular values of $\bm{C}$.

% 6. Numerical Methods
\section{Numerical Experiments}\label{sec:Numerical_experiments}
This section describes numerical experiments conducted in one and two dimensions.
We utilize Python libraries such as NumPy  \cite{NumPy} and SciPy \cite{2020SciPy-NMeth}.
Additionally, the FEniCSx platform   \cite{FEniCSX}  is utilized for PDE discretization and solution in the 2D case. 
The computations were performed on a 
2020 Apple MacBook Air (M1, 8GB RAM).
%
%%%%%%%%%%%%%%%%%%%%%%%%%%%%%%%%%%%%%%%%%%%%%%%%%%%%%%%%%%%%%%%%%%%%%%%%%
\subsection{1D heat equation}\label{sec:1d_problem} 
%%%%%%%%%%%%%%%%%%%%%%%%%%%%%%%%%%%%%%%%%%%%%%%%%%%%%%%%%%%%%%%%%%%%%%%%%
The first experiment focuses on the time-dependent heat equation in one spatial dimension. This experiment investigates the spatial and temporal evolution of the temperature distribution $u$ within a domain $\Omega:=(0,1)$. The heat equation governing this evolution can be expressed as follows:
\begin{align}\label{eqn:1dheat}
\begin{aligned}
\frac{\partial u}{\partial t} &= 
\frac{\partial}{\partial x}  \left({\kappa} \frac{\partial u}{\partial x}   \right), & \text{in } \Omega \times (0, T), \\
u &= 0, & \text{on } \partial \Omega \times (0, T), \\
u &= u_0, & \text{on } \Omega \times \{0\}.
\end{aligned}
\end{align}
Here, $u(x, t)$ denotes the temperature at spatial position $x$ and time $t$, while ${\kappa}$ represents the thermal diffusivity.  
The inverse problem involves estimating the initial condition $u_0$ from discrete measurements of the state $u(x,t)$ in space and time at selected sensor locations.   
\subsubsection*{Model error and background}

The ``true model'' is represented by~\eqref{eqn:1dheat} with thermal diffusivity:
\[{\kappa}^\epsilon ({x}) = 2 + \sin\left({2\pi (x/ \epsilon)}\right) \quad \forall x \in \Omega,
\]
with $\epsilon = 2^{-4}$, representing the microstructural length scale.  To introduce modeling error, we apply homogenization, a standard technique for handling inhomogeneous microstructures, and set the homogenized diffusivity to  ${\kappa}^0(x) \equiv \sqrt{3}$, see  \cite{ALEXANDERIAN_HOMOGENIZATION, bensoussan1978asymptotic}. For the WC4D-Var, we use the homogenized thermal diffusivity in the data assimilation. We take  the  initial conditions  for the true and background models as discretized representations of the following
\begin{align}
\begin{aligned}
    u_0^{\text{true}}(x)  := & \> \exp\left(-\frac{1}{2}\left(\frac{x - \mu}{\sigma}\right)^2\right), \quad \mbox{ where } \mu = 0.7 \mbox{ and } \sigma = 0.08.  \\
    u_0^{\text{b}}(x)  := & \> u_0^{\text{true}}(x) + 0.2\exp\left(-\frac{1}{2}\left(\frac{x - \mu^b}{\sigma}\right)^2\right), \quad \mbox{ where } \mu^b = 0.2.  
\end{aligned}    
\end{align}
Thus, the background $u_0^{\text{b}}$ is a perturbation of the true initial condition $u_0^{\text{true}}$.

\subsubsection*{Discretization}
To solve the PDE~\eqref{eqn:1dheat}, we employ the finite element method (FEM) in space and the implicit Euler in time discretization. 
More specifically, we consider uniform partitions for both time and space, where the spatial mesh size is $h$ and the time step is denoted by $\Delta t$. 
In particular, the spatial domain $\Omega$ is discretized using a uniform mesh with 400 intervals.
Let $\{ \varphi_k\}_{k =1}^{d_S}$ represents a FEM basis for $\mathcal{P}_1(\Omega_h)$, where $\Omega_h$ is a mesh/partition of $\Omega$,  and $d_S$ represents the number of basis functions. The mass and stiffness matrices  are defined component-wise as:
\begin{align*}
    \mathbf{N}_{i,j}&:=  \int_{\Omega} \varphi_i\varphi_j \,dx, \quad 
    \mathbf{K}_{i,j}:= \int_{\Omega}   \kappa(x) 
    \frac{\partial \varphi_i}{\partial x} 
    \frac{\partial \varphi_j}{\partial x}
    \,dx, \qquad 1 \leq i,j  \leq d_S,
\end{align*}
where $\kappa$ is either $\kappa^\epsilon$ for the true or $\kappa^0$ for the inaccurate model. 

Denoting $\bmu_m\in \R^{d_S}$ as the coefficients of the FEM approximation for $u(\cdot,m*\Delta t)$, the discrete heat equation becomes 
\begin{align}
    \bmu_{m+1} &= (\mathbf{N}+\Delta t \mathbf{K})^{-1} {\mathbf{N}}\bmu_m, \quad \mbox{ for }  0\le m \le m_T-1, \label{eq:evolution}
\end{align}
where 
$m_T$ is the number of time steps for the time integration.
We set the final time to $T = 4 \times 10^{-2}$ with a discretization time step of $\Delta t \approx 1.5625 \times 10^{-5}$. Observation times are given by 
\[
t_\ell = k \Delta t, \quad \ell = 1,\dots,n_T,
\]
with $k \in \mathbb{N}$ chosen so that $t_k$ are uniformly spaced at intervals of $4 \times 10^{-3}$, yielding $n_T = 10$ snapshots.

\subsubsection*{Data acquisition} 
We consider $n_s = 28$  candidate sensors placed uniformly in the interval $[0.025,0.975]$, see \Cref{fig:sensorlocationsRandomGKS1D} along with designs selected by the proposed algorithms; the details are explained below. Given the use of piecewise linear elements with $d_S = 401$, the sensor locations define the sparse observation matrix \( \O_\ell \) of size \( n_s \times d_S \), which remains fixed in time for \( 1 \leq \ell \leq n_T \).  
We use {synthetic observations} generated using the true model with $2\%$ additive Gaussian noise to simulate measurement error.
\subsubsection*{Background and model uncertainty} The background is taken to be
$\bm{u}_0  \sim \mathcal{N}( \bm{u}_0^{b}, \Gb)$, where $\bm{u}_0^{b}$ is the discretized representation of $u_0^b$.
  We take the background covariance matrix  
\begin{align}
\Gb = \left(\gamma\mathbf{K} + \mathbf{N}\right)^{-1}\mathbf{N} \left(\gamma\mathbf{K} +\mathbf{N}\right)^{-1},
\end{align}
where we set  $\gamma=10^{-1}$. To construct the covariance matrices \(\bm Q_\ell\) for \(1 \leq \ell \leq n_T\), we approximate the model error covariance using sample-based estimation. Specifically, we generate \(N = 40\) trajectories by sampling initial states from the  prior distribution $\mathcal{N}( \bm{u}_0^{b}, \Gb)$,  and then evolving them using both the true and approximate (background) models. For each time step \(\ell\), the difference between the true and approximate trajectories forms a collection of error samples. We compute the sample covariance matrix \(\widetilde{\bm \Gamma}_\ell\) of these errors at each time step. 
To ensure numerical stability and avoid issues caused by small singular values, we add a small regularization term,  ``nugget'', to the diagonal. The resulting covariance matrix is given by:
\[
\bm Q_\ell = \widetilde{\bm \Gamma}_\ell + \delta_\ell \I, \qquad 1 \leq \ell \leq n_T,
\]
where \(\delta_\ell = 10^{-12} + 10^{-6} \|\widetilde{\bm \Gamma}_\ell\|_2\). This regularization ensures that \(\bm Q_\ell\) remains well-conditioned while smoothly damping the smaller eigenvalues.
\subsubsection*{The MAP estimate}
We revisit the linear system  \eqref{def:distribution_upost} related to the MAP estimate:
\begin{align*}
    \left( \O\t \Gmeas^{-1} \O + \L\t \Gmod^{-1} \L\right) \upost
&= 
    \O\t \Gmeas^{-1} \bmyo + \L\t \Gmod^{-1} \L \uprior
\end{align*}
Using the closed-form expressions for $\L^{-1}$ and its transpose from \eqref{lemma:L}, we use  $\Gpr = (\L\t \Gmod^{-1} \L)^{-1}$, cf. \eqref{def:distribution_uprior},  as a preconditioner for the iterative solvers.
The use of this particular preconditioner significantly reduces the required iterations. The number of Preconditioned Conjugate Gradient (PCG) iterations is reduced from $1409$ for the unpreconditioned case (i.e., identity as a preconditioner) to just 108 for the preconditioned case.

\subsubsection*{Estimating the EIG}

Here, we present results for matrix-free methods to approximate different formulations of the EIG, as discussed in \Cref{section:Alternative_formulations} and \Cref{sec:preclanc}. The experiments are divided into two parts.
First, we assess the accuracy of the SLQ method for log-determinant (EIG) approximations across four WC-4DVar formulations. In the second part, we compare the SLQ with the approach that employs the \xntrace{} method combined with Lanczos iterations.

We compare the different formulations in Section~\ref{section:Alternative_formulations}. We limit the number of sampling vectors  to $N=2^3$ and use a relative error tolerance of $1 \times 10^{-10}$ as the stopping criterion for the Lanczos iterations. The results are presented in \Cref{tab:logdet_4form}, where the Lanczos iterations are averaged over the $N$ samples. We see that the Preconditioned formulation has the lowest relative error for the fixed sample size; however, it has a higher computational cost (see \Cref{tab:formulations_comparison}).  

\begin{table}[!ht]
\centering
\begin{tabular}{lcccc}
\toprule
\textbf{Formulation} & \textbf{Matrix size} $d$ & \textbf{Lanczos Iters} &  \textbf{Rel. Error} \\
\midrule
\textbf{Unpreconditioned} \eqref{eqn:unprec}  & $4010$ & 533 &  $7.9 \times 10^{-4}$ \\
\textbf{Preconditioned} \eqref{eqn:criterion} & $4010$ & 37 &   $1.5 \times 10^{-4}$\\
\textbf{SP-I} \eqref{eq:Saddle1} & $8300$ & 350 &   $1.1 \times 10^{-4}$\\
\textbf{SP-II} \eqref{eq:Saddle2} & $8020$ & 550 &$9.0 \times 10^{-2}$ \\
\bottomrule
\end{tabular}
\caption{Results for Log-Determinant Approximation using SLQ with $N=2^3.$ }
\label{tab:logdet_4form}
\end{table}

Next, we compare SLQ with \xntrace{} coupled with Lanczos iterations for the preconditioned formulation.  
Table \ref{tab:SLQ_summary} summarizes the mean relative error, standard deviation, and average number of iterations for Lanczos iterations need for SLQ and \xntrace{} estimators. We see that \xntrace{} is more accurate for the same number of samples, with a lower mean relative error and standard deviation.

\begin{table}[ht!]
\centering
\begin{tabular}{cc|cc|cc}
\hline
\textbf{Samples} & \textbf{Lanczos}  & \multicolumn{2}{c}{\textbf{SLQ}} & \multicolumn{2}{c}{\textbf{\xntrace{}+Lanczos} } \\
\cline{3-6}
   $(N)$     &      \textbf{Iters.}         & \textbf{Mean} & \textbf{Std. Dev.} & \textbf{Mean} & \textbf{Std. Dev.} \\
\hline
2       & 73.5                & $2.99 \times 10^{-4}$ & $2.49 \times 10^{-4}$ & $4.17 \times 10^{-4}$ & $2.99 \times 10^{-4}$ \\
4       & 148.2               & $2.08 \times 10^{-4}$ & $1.62 \times 10^{-4}$ & $2.40 \times 10^{-4}$ & $1.85 \times 10^{-4}$ \\
8       & 296.3               & $1.48 \times 10^{-4}$ & $1.28 \times 10^{-4}$ & $1.80 \times 10^{-4}$ & $1.48 \times 10^{-4}$ \\
16      & 592.1               & $1.01 \times 10^{-4}$ & $8.08 \times 10^{-5}$ & $1.11 \times 10^{-4}$ & $7.85 \times 10^{-5}$ \\
32      & 1186.9              & $8.31 \times 10^{-5}$ & $5.48 \times 10^{-5}$ & $5.59 \times 10^{-5}$ & $4.11 \times 10^{-5}$ \\
64      & 2369.5              & $4.85 \times 10^{-5}$ & $4.58 \times 10^{-5}$ & $4.01 \times 10^{-6}$ & $3.18 \times 10^{-6}$ \\
128     & 4741.9              & $3.75 \times 10^{-5}$ & $2.90 \times 10^{-5}$ & $1.05 \times 10^{-7}$ & $8.24 \times 10^{-8}$ \\
\hline
\end{tabular}
\caption{Number of samples, average Lanczos iterations, average relative errors, and standard deviations for Hutchinson's  and \xntrace{}  estimators over 100 trials.}
\label{tab:SLQ_summary}
\end{table}

\subsubsection*{Sensor placement} 
This section evaluates the performance of deterministic and randomized sensor placement methods introduced in Section \ref{sec:sensor_placement}. 
Our aim is to select a subset of sensors from  $n_s=28$ candidate sensor locations. We consider two cases: selecting $k=10$ and $k=5$ sensors. 
An exhaustive search for  $k=10$  would require evaluating   ${28 \choose 10} \approx 13 \times 10^6$ possible combinations, which is computationally prohibitive for large problems.
The columns selected by the deterministic approach based on the CSSP approach using the GKS method are denoted by $\A\SSS{CSSP}$. In contrast, the RAF-OED, which employs a sketching matrix with an oversampling parameter $p=20$, see \eqref{eq:Rand_sensor_selection}, yields columns denoted by $\A\SSS{R}$. 
Additionally, we include results from a  greedy selection algorithm denoted by $\A\SSS{G}$, for comparison. 

\begin{table}[!ht]
    \centering
    \caption{EIG Values from Different Methods }
    \label{tab:d_optimal_values}
    \begin{subtable}[t]{0.5\textwidth}
        \centering
        
        \begin{tabular}{lr}
            \toprule
            Method & {$\criteria(\S)$} \\
            \midrule
            Best ($\S_\textrm{max}$) & 106.30 \\
            Greedy ($\S=\S_\text{G}$) & { 95.76} \\
            Algorithm~\ref{alg:GKS_sensorPlacement} & {97.71} \\
            RAF-OED & 97.71 \\
%            Modes  Designs & 94--102 \\
            \bottomrule
        \end{tabular}
        \caption*{$k=5$}
    \end{subtable}%
    \hfill
    \begin{subtable}[t]{0.5\textwidth}
        \centering
        \begin{tabular}{lS[table-format=3.3]}
            \toprule
            Method & {$\criteria(\S)$} \\
            \midrule
            Best ($\S_\textrm{max}$)  & unknown \\
            Greedy ($\S=\S_\text{G}$) & 127.31 \\
            Algorithm~\ref{alg:GKS_sensorPlacement}  & 125.89 \\
            RAF-OED & 127.22 \\
            %Median & 120.001 \\
            \bottomrule
        \end{tabular}
        \caption*{$k=10$}
    \end{subtable}
\end{table}

\begin{figure}[!ht]
    \centering
    \includegraphics[width=1\linewidth]{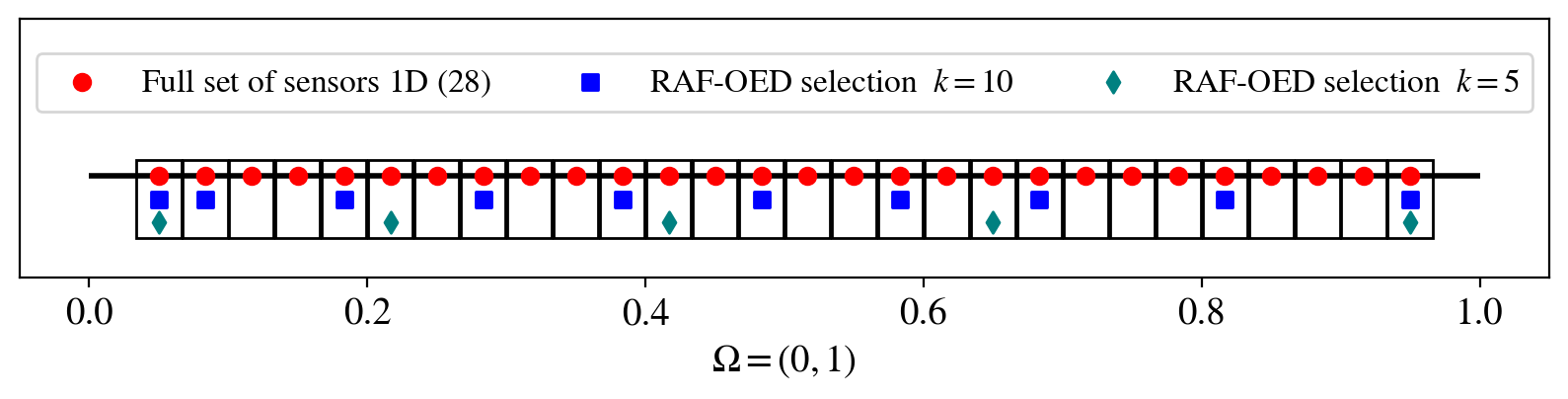}
    \caption{The sensor locations determined by RAF-OED for different values of $k$.}
    \label{fig:sensorlocationsRandomGKS1D}
\end{figure}
\begin{figure}[!ht]
    \centering
    \begin{subfigure}[b]{0.48\linewidth}
        \centering
        \includegraphics[width=\linewidth]{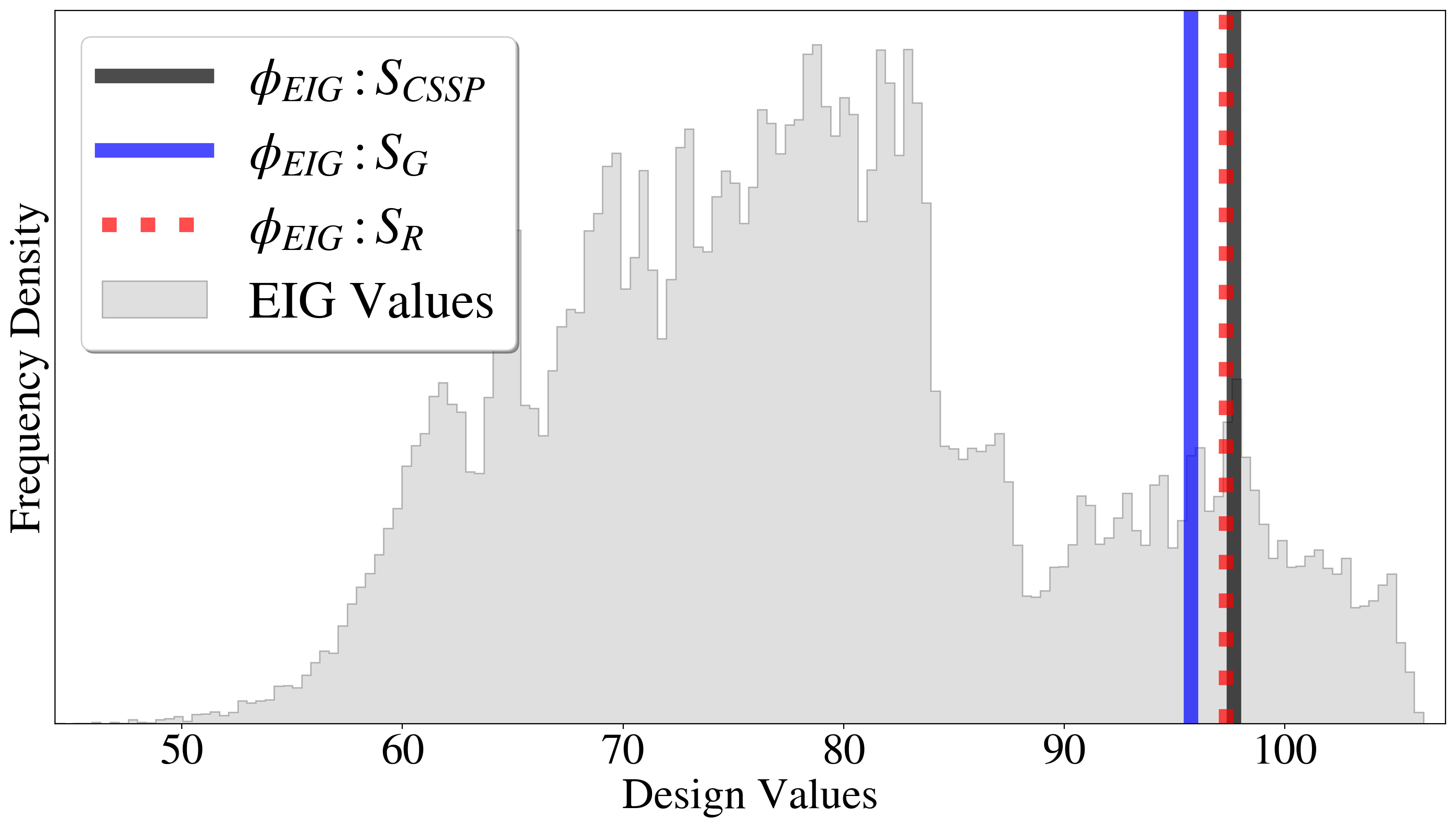}
        \caption*{$k=5$}
        \label{fig:OED_rand_k5}
    \end{subfigure}%
    \begin{subfigure}[b]{0.48\linewidth}
        \centering
        \includegraphics[%height=4cm,
        width=\linewidth]{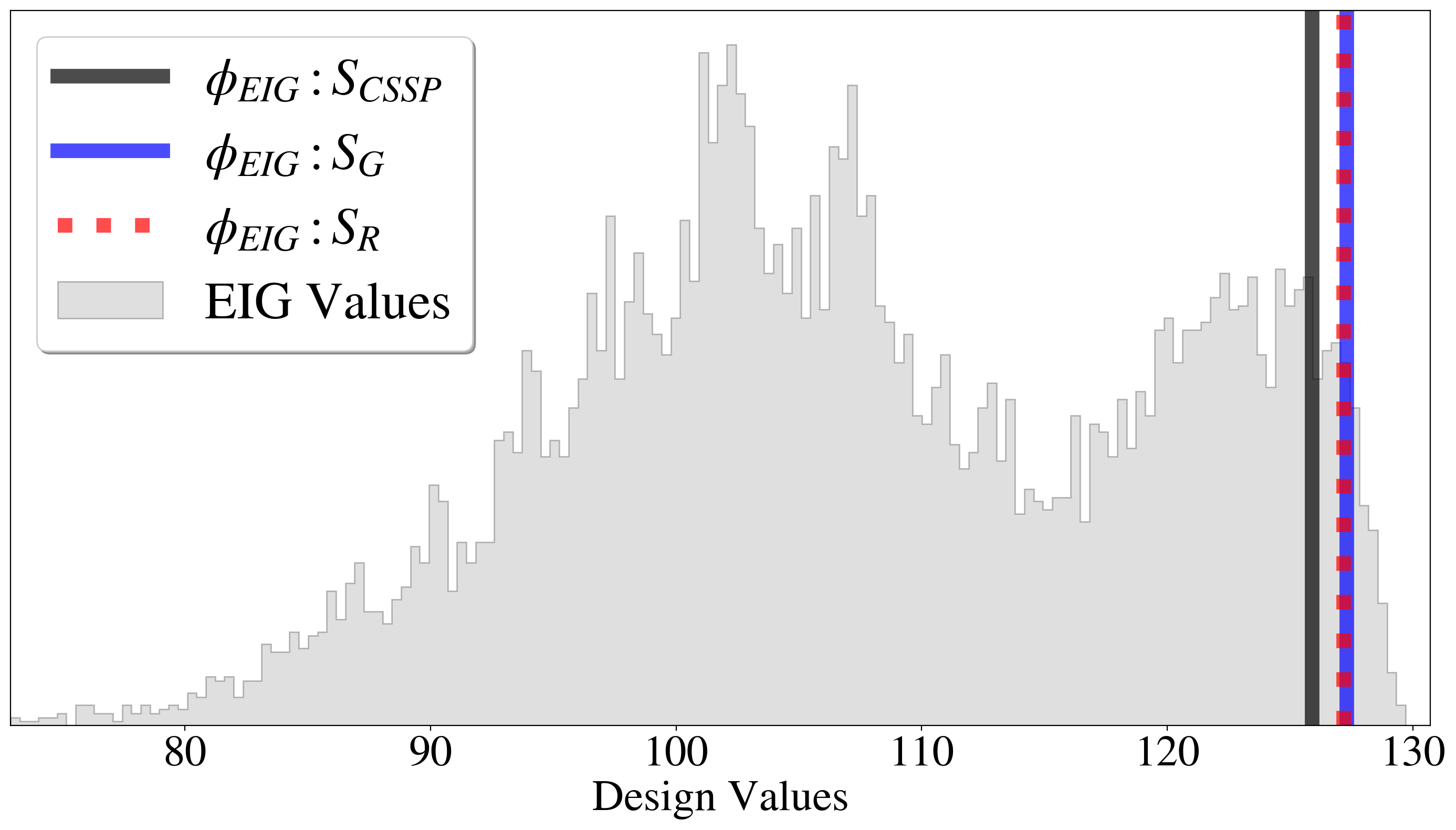}
        \caption*{$k=10$}
        \label{fig:histo_k10}
    \end{subfigure}
    \caption{Comparison of different sensor placement methods. The histogram represents the criterion values for random designs.}
    \label{fig:CS}
\end{figure}

Figure~\ref{fig:sensorlocationsRandomGKS1D} shows the locations of the full set of sensors ($n_s=28$), along with those selected by the randomized approach ($\SSS{R}$) for both $k=5$ and $k=10$. 
Figure~\ref{fig:CS}  compares the design values obtained using these methods along with \(6 \times 10^4\) random designs  for $k=10$, and for $k=5$ we considered all ${28 \choose 5 } = 98,280$ possible columns combinations. 
For $k=10$, the GKS-based method yields a higher EIG value compared to approximately 94\% of the random designs, while its randomized variant (RAF-OED) exceeds about 97\% of the designs.
When selecting \(k=5\) sensors, both methods (GKS and RAF-OED) achieve better results than 91\% of all possible designs. 
Notably, in both scenarios, our algorithm achieves results that are either comparable to or better than those of the greedy approach. As noted in~\cite{eswar2024bayesian}, the greedy approach is more expensive compared to the CSSP methods discussed. 
Therefore, the performance of our algorithm is particularly favorable, as it delivers similar results without some of the computational limitations of the greedy approach.

\subsection{2D advection-diffusion problem}
We consider a two-dimensional advection-diffusion (AD) problem from~\cite{VillaPetraGhattas21}, where the concentration $c= c(\bm x,t) $ is governed by the following partial differential equation (PDE):
\begin{align}
\begin{aligned}
    \frac{\partial c}{\partial t} + \bmv\cdot \nabla c- \mathrm{div} \nabla c &= f \qquad \mbox{in } \Omega:= (-1,1)^2,\; t \in (0,2], \\ % \bm \epsilon 
    c(0)&= c_0.
\end{aligned}
    \label{eq:AD}
\end{align}
Here we assume homogeneous Neumann boundary conditions. We also consider a divergence-free velocity field, visualized in Figure \ref{fig:v_field}, denoted as 
$$\bm v(x,y):=[2y(1-x^2)\,\, -\!2x(1-y)^2]\t,$$ 
which is  a simplification of a solution for the cavity flow problem. Furthermore, $f$ is a source term. As in the previous application, the goal is to determine the initial condition $c_0$ from discrete spatiotemporal measurements of the field $c(\bm{x},t)$.

\subsubsection*{True and background models}
Here, we consider uncertainties arising from the initial conditions.  Specifically, for the true initial condition, which results in the true state/snapshots $\bm c^\text{true}_\ell$,  we assume that the concentration is zero throughout the domain, except in two localized regions. These regions exhibit concentration patterns characterized by Gaussian-like distributions.

 \begin{figure}[htb!]
    \centering 
    \includegraphics[width=0.28\linewidth]{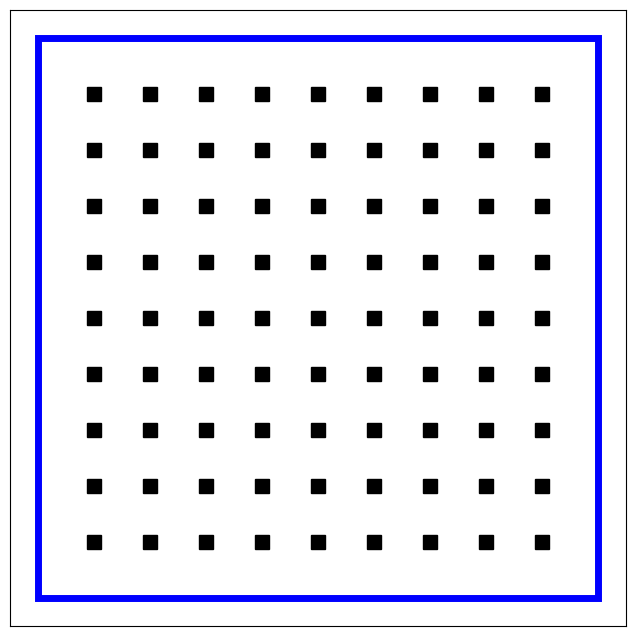}
    \includegraphics[width=0.28\linewidth]{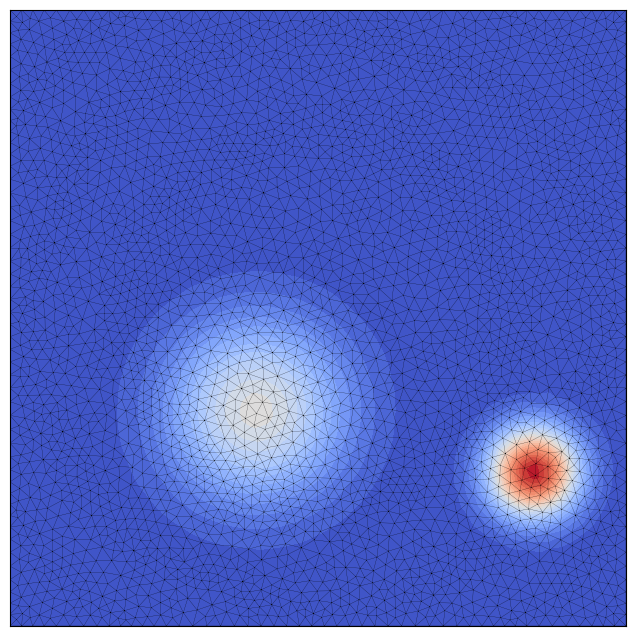}
     \includegraphics[width=0.38\linewidth]{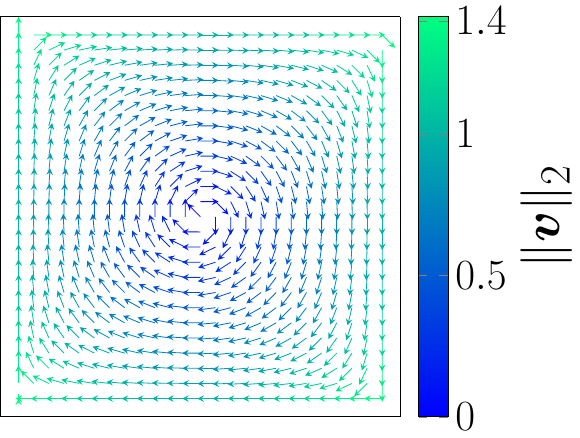}
    \caption{Figures illustrating the candidate sensor locations (left), initial condition $ c_0$ (middle), and velocity field $\bmv$ (right).}
    \label{fig:v_field}
\end{figure} 

For data collection, we employ a similar approach as in the 1D problem. However, in this case, the sensor locations are uniformly distributed on a 2D grid, positioned away from the boundary.
 In time, we discretize the interval $[0,2]$ into 200 steps using a uniform time step $\Delta t$. Observations are recorded every 20 time steps, resulting in 10 snapshots.
Figure  \ref{fig:v_field} shows the sensor distribution (left), the true initial concentration $c_0$ (middle), and the velocity field $\bmv$ (right).

\subsubsection*{Discretization}

We consider the Backward Euler method for temporal discretization and a primal formulation for spatial discretization of (\ref{eq:AD}). Specifically, we consider the Lagrange elements of order 1 denoted by $V_h$. 
Thus,  given $c_n \approx c(\cdot,t_n)$, we seek $c_{n+1}\in V_h$ such that:
\begin{align}
\int_\Omega \left(\frac{c_{n+1} - c_{n}}{\Delta t} \right)w_h + (\bmv_h \cdot \nabla c_{n+1})w_h + \left( \nabla c_{n+1} \cdot \nabla w_h \right) d\bm x  = \int_\Omega f_{n+1}w_h d\bm x, \label{eq:formulation1AD}
\end{align} 
for all $ w_h $ in $V_h.$
Here, we have assumed a uniform time partition and homogeneous Neumann boundary conditions for simplicity. 
This method leads to the following linear system:
\begin{align}
    \begin{aligned}
        \left(\mathbf{N}+\Delta t\mathbf{B}_{\bm v_h} + \Delta t\mathbf{K}\right)\bm c_{n+1} = \mathbf{N} \left( \Delta t \bm f_{n+1} + \bm c_n\right),
    \end{aligned}  \label{eq:formulation1AD_discrete}
\end{align}
where $ 0 \le n  \le 200$, and observations are made  every 20 time steps. 
Here, $\mathbf{N}$ and $\mathbf{K}$ denote the mass and stiffness matrices, respectively, while $\mathbf{B}_{\bm v_h}$ represents the matrix associated with the  (non-symmetric) bilinear form:
\begin{align*}
    \mathcal{B}(c_h,w_h;\bm v_h):=\int_\Omega (\bmv_h \cdot \nabla c_h)w_h d\bm x\qquad \forall c_h,w_h\in V_h.
\end{align*}
We employ a sparse LU factorization to solve the system (\ref{eq:formulation1AD_discrete}), and  $\bmv_h$ is the $L^2$-projection of $\bmv$ onto $\mathcal{P}^2_2(\mathcal{T}_h)$ space.  
The FEM approximation of $c$ at equispaced times over $[0,T]$ is depicted in Figure \ref{fig:c_snapshots}.
 \begin{figure}[ht!]
    \centering
    \includegraphics[width=0.15\linewidth]{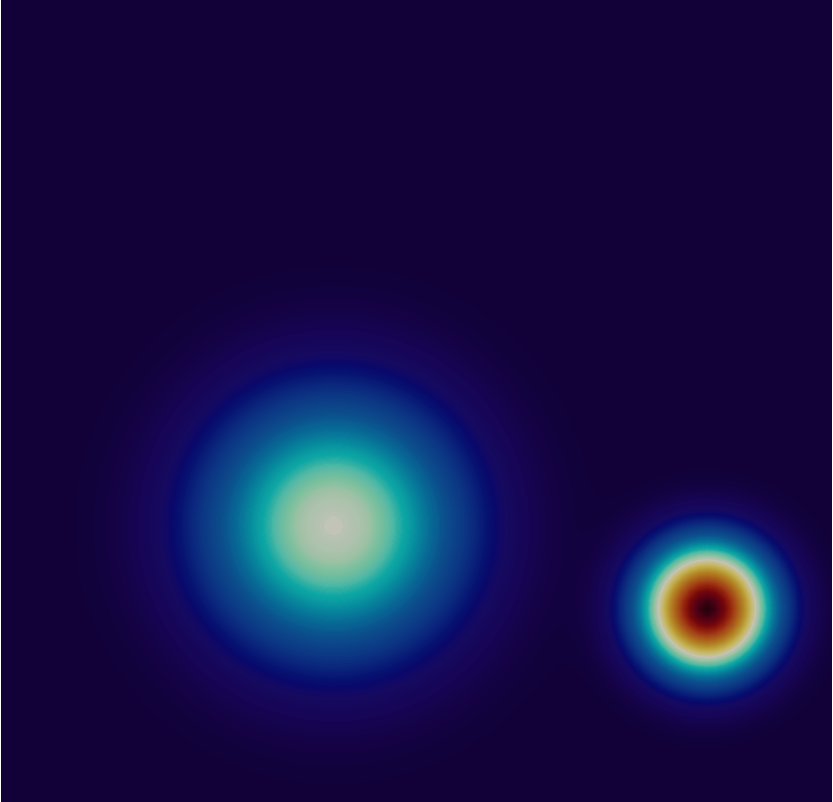}
    \includegraphics[width=0.15\linewidth]{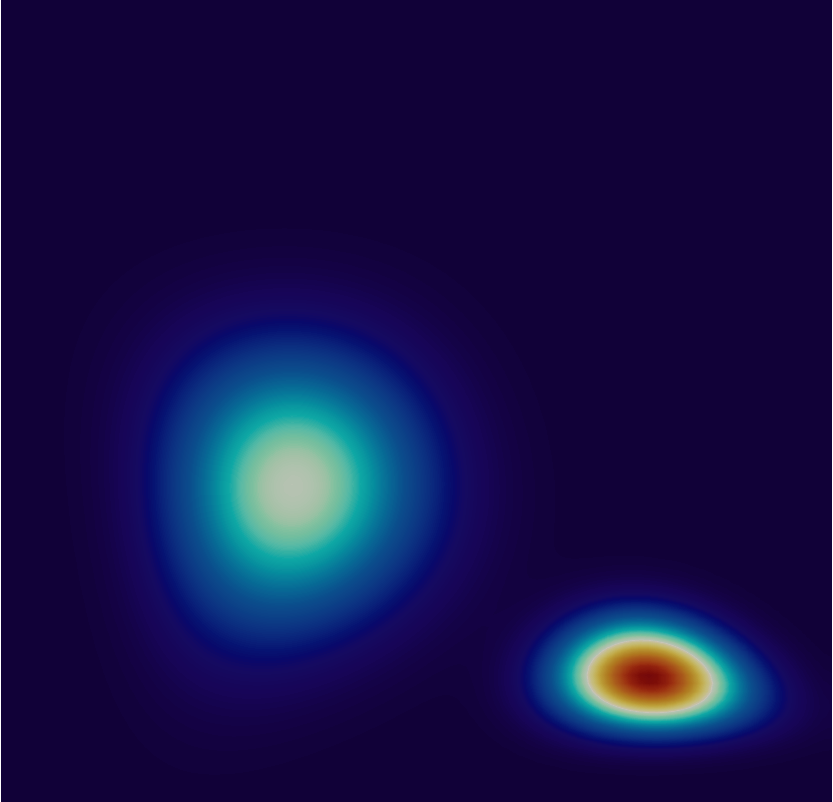}
    \includegraphics[width=0.15\linewidth]{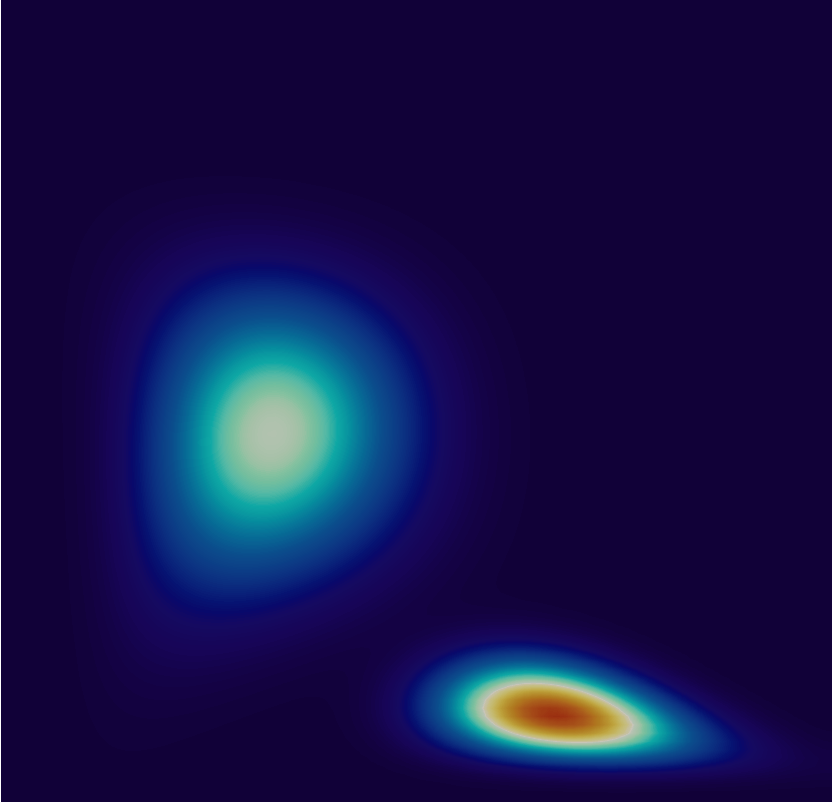}
    \includegraphics[width=0.15\linewidth]{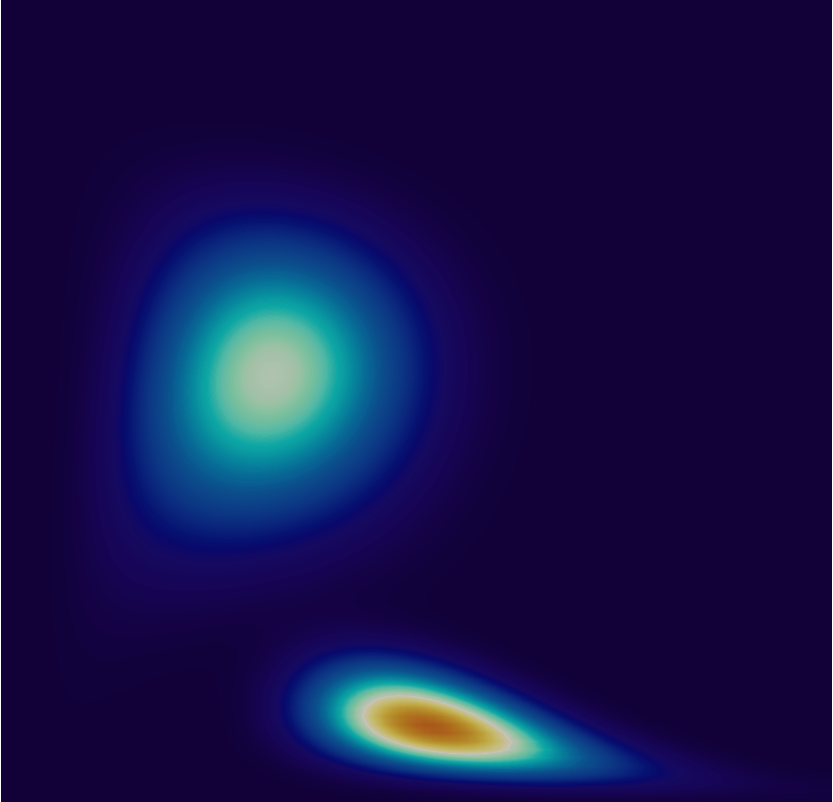}
    \includegraphics[width=0.15\linewidth]{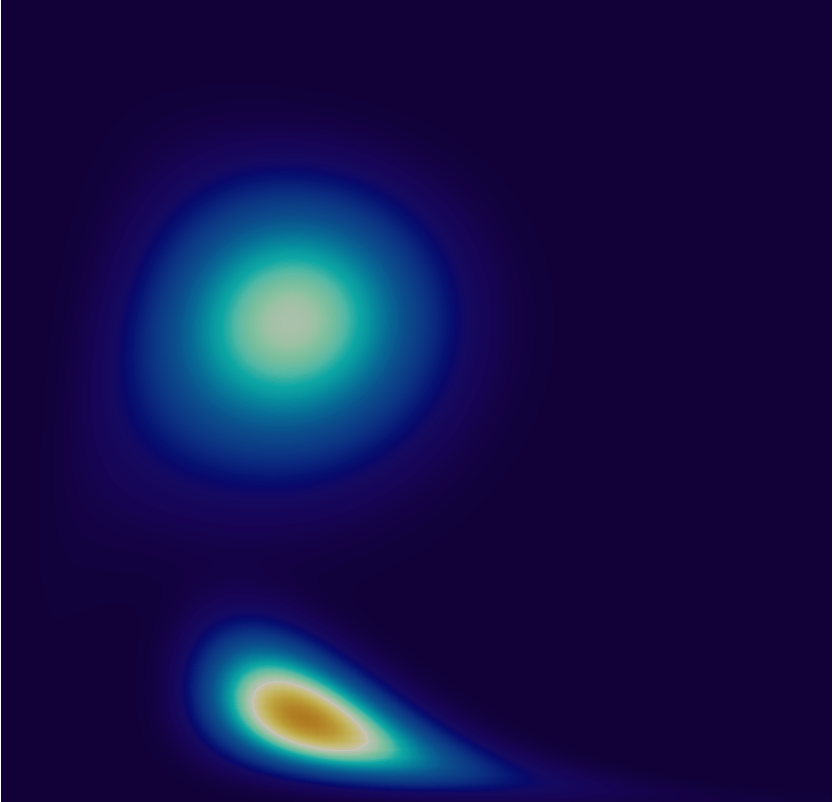}
    \includegraphics[width=0.15\linewidth]{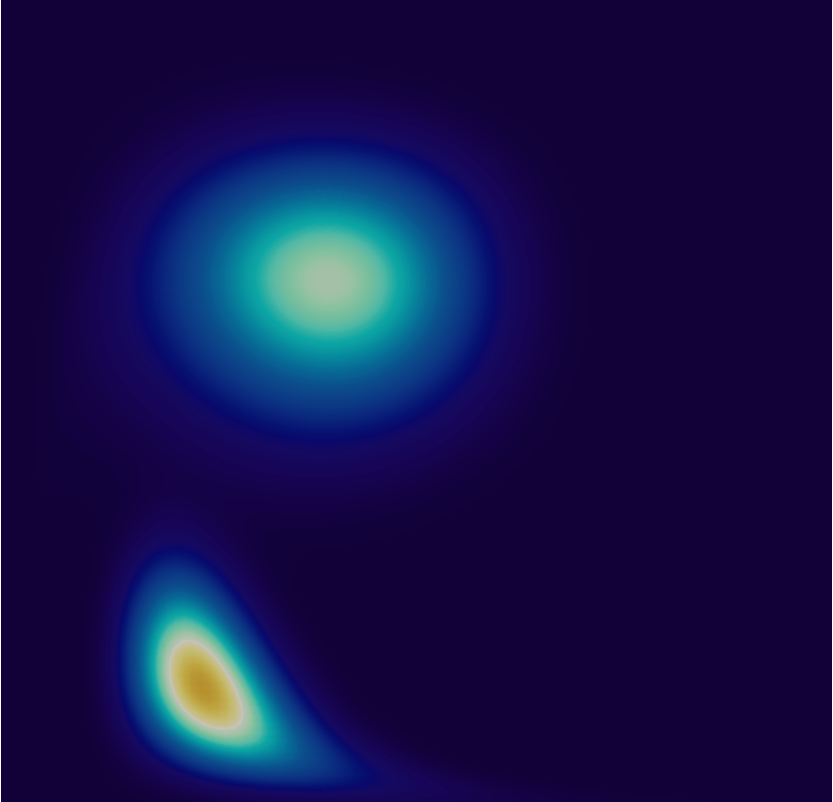}
    \includegraphics[width=0.15\linewidth]{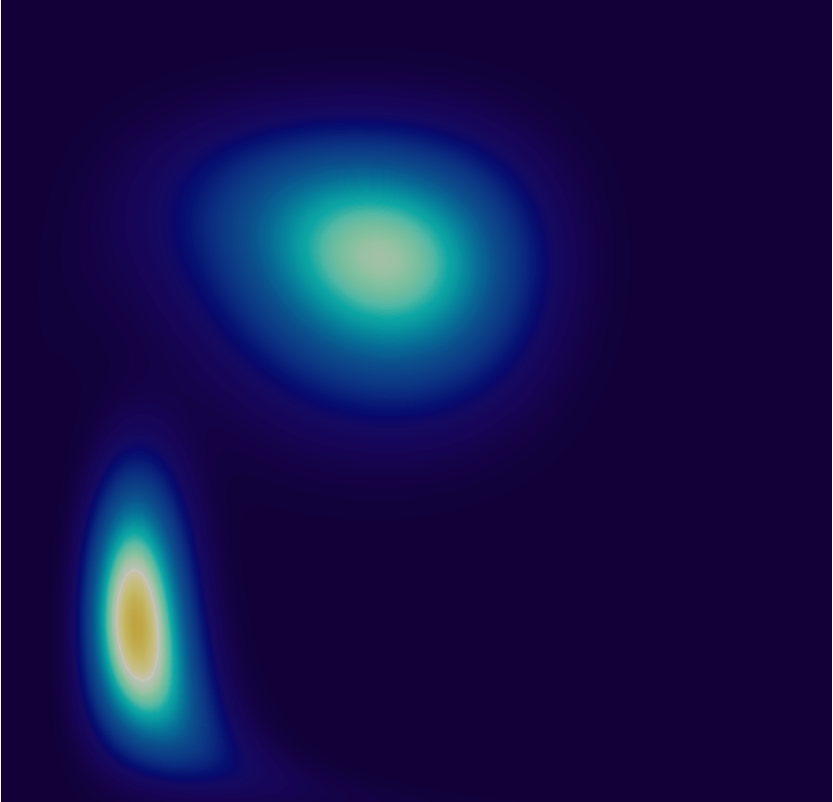}
    \includegraphics[width=0.15\linewidth]{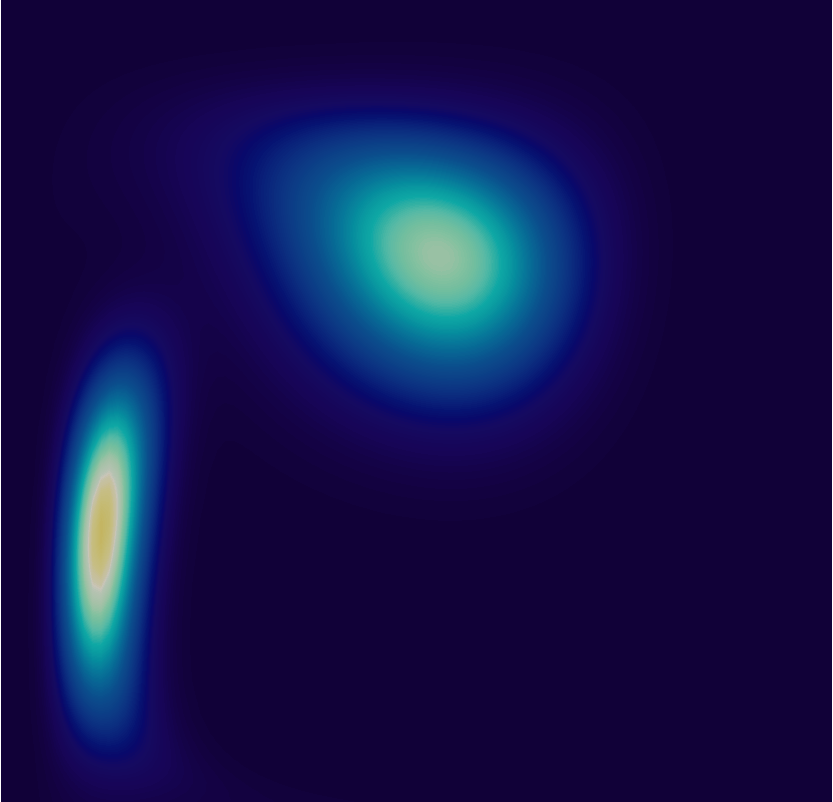}
    \includegraphics[width=0.15\linewidth]{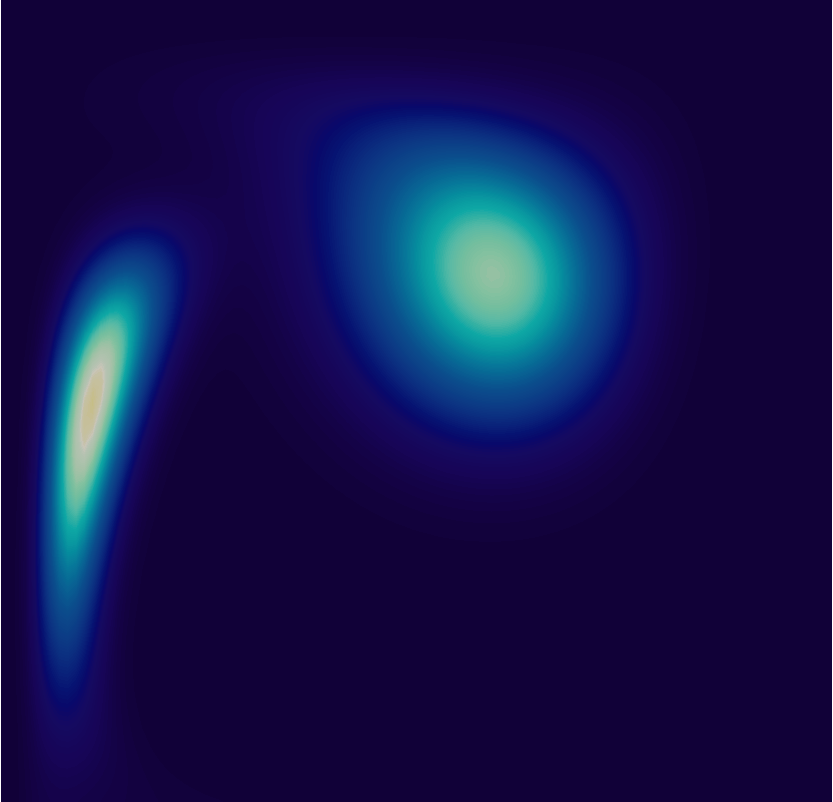}
    \includegraphics[width=0.15\linewidth]{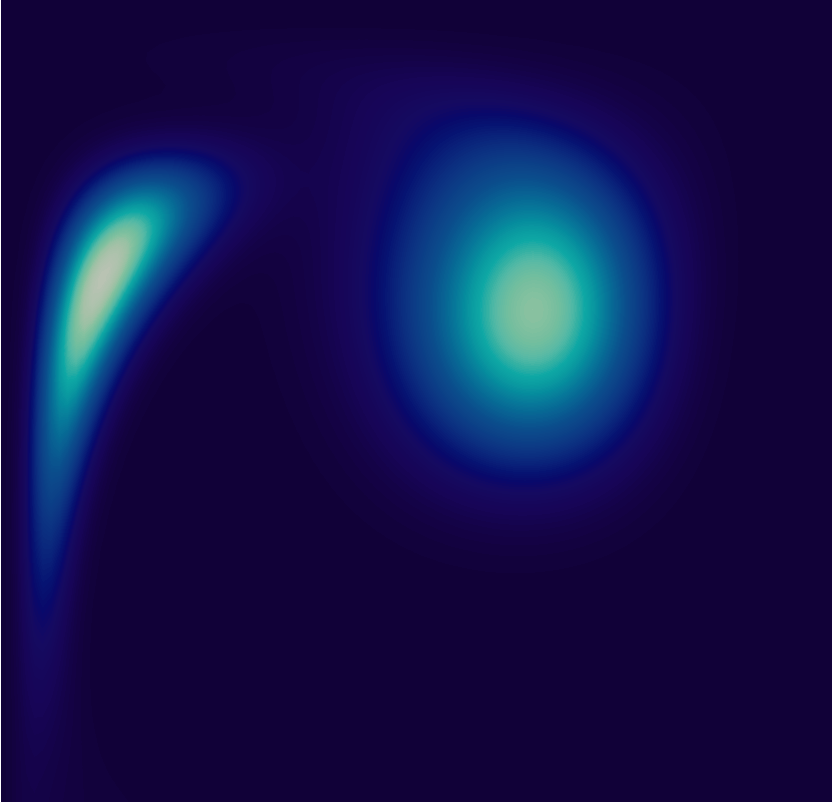}
    \includegraphics[width=0.15\linewidth]{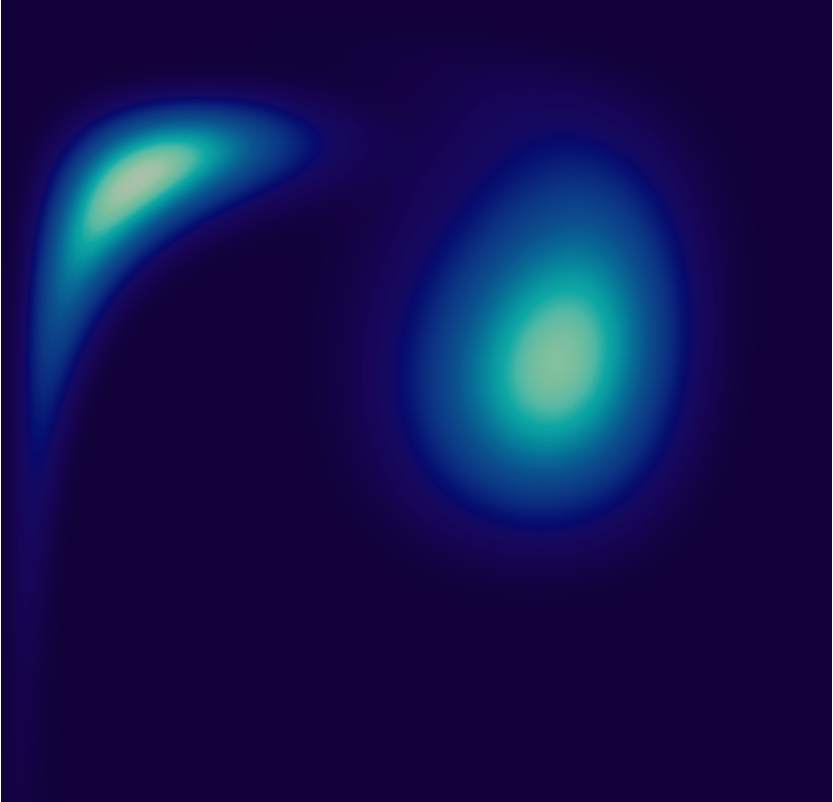}
    \caption{Snapshots of $c_h$. }
    \label{fig:c_snapshots}
\end{figure} 
Here we consider $n_T=10$, $f\equiv 0$, and the evolution operator $\Map{\ell}{\ell+1}$ advances the discrete system by 20 time steps for each $ 0 \le \ell \le n_T-1$. 

\subsubsection*{Background and model error}
We consider a Gaussian prior for the (discrete) initial condition $\bm c_0$ as $\mathcal{N}(\overline{\bm c}, \Gb)$. Following \cite{Stuart_2010, VillaPetraGhattas21},  we define the covariance matrix $\Gb$ using an elliptic operator as follows:
\begin{align} \label{eq:prior_2d}
\Gb = \left(\gamma\mathbf{K} + \delta\mathbf{N}\right)^{-1}\mathbf{N} \left(\gamma\mathbf{K} +\delta\mathbf{N}\right)^{-1},
\end{align}
where $\gamma = 2.70$ and $\delta = 2.5$ are parameters controlling variance and correlation length of the prior. To simulate model error,  we take $\displaystyle  \bm\eta_\ell \sim \mathcal{N}(\bm 0, \bm Q_{\ell})$, where $\bm{Q}_{\ell} = q^2_\ell \I$ and $q_\ell := 0.05\frac{1}{\sqrt{n_{s}}}  \| \Map{0}{\ell} \bm c^{\text{true}}_0 \|_2$ for $1  \le \ell \le n_T$.

\subsubsection*{EIG estimation}

We now evaluate the EIG for the Preconditioned formulation using SLQR and \xntrace. Note that for this example, we do not have the ground truth, so we do not report the errors but the statistics of the estimates.   
The number of Lanczos iterations is $70$, on average. The mean of the trace estimators,  the relative standard deviation (RSD), and the standard deviation (SD), is summarized in Table \ref{tab:estimator_comparison}.

\begin{table}[ht!]
\centering
\begin{tabular}{c|ccc|ccc}
\hline
\textbf{$N$} & \multicolumn{3}{c|}{\textbf{SLQ}} & \multicolumn{3}{c}{\textbf{\xntrace{}}} \\ 
\textbf{} & \textbf{Mean} & \textbf{RSD} & \textbf{SD} & \textbf{Mean} & \textbf{RSD} & \textbf{SD}\\ \hline
2 & $4.1612 \times 10^{5}$ & $8.83 \times 10^{-2}$\%  & 368& $3.9789 \times 10^{5}$ & $2.96 \times 10^{-2}$\% & 118\\ \hline
4 & $4.1612 \times 10^{5}$ &  $6.75 \times 10^{-2}$\% & 281& $3.9790 \times 10^{5}$ & $2.29 \times 10^{-2}$ \% & 91\\ \hline
8 & $4.1609 \times 10^{5}$ &  $5.15 \times 10^{-2}$\% & $\bm{214}$& $3.9788 \times 10^{5}$ & $1.56 \times 10^{-2}$\% & $\bm{62}$\\ \hline
\end{tabular}
\caption{Comparison of SLQ and \xntrace{} over 100 trials.} 
\label{tab:estimator_comparison}
\end{table}

\begin{figure}[!ht]  
  \centering
  \begin{subfigure}[b]{0.45\linewidth}
    \centering
    \includegraphics[width=\linewidth]{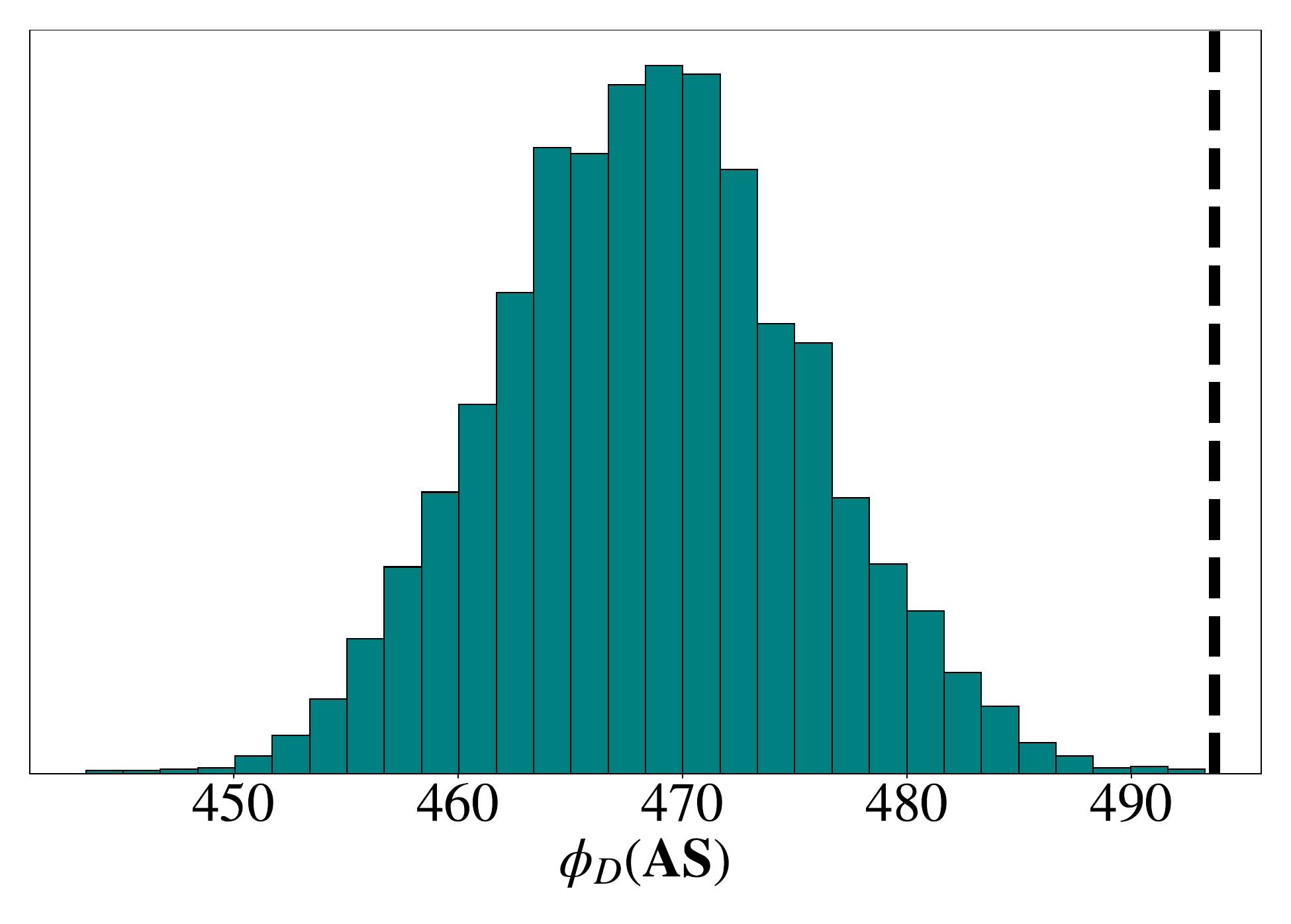}
    \caption*{Design values for $k=10$ sensors}
    \label{fig:k10}
  \end{subfigure}
  \hfill
  \begin{subfigure}[b]{0.45\linewidth}
    \centering
    \includegraphics[width=\linewidth]{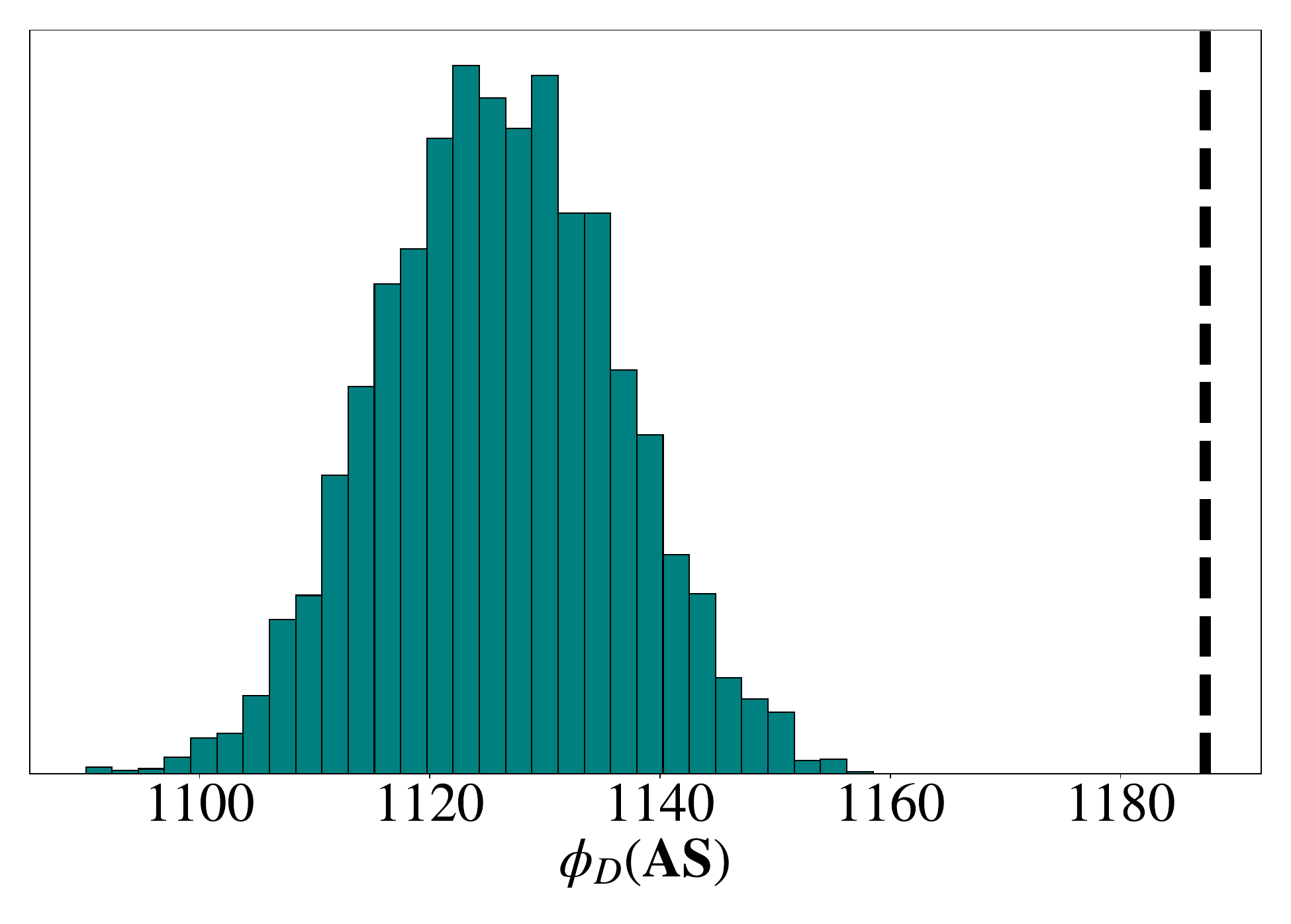}
    \caption*{Design values for $k=25$ sensors}
    \label{fig:k25}
  \end{subfigure}
  \caption{Comparison between the RAF-OED (dashed black) and the random designs (histogram). }
  \label{fig:CS2D}
\end{figure}

\subsubsection*{Sensor placement}

We consider the problem of selecting an optimal subset of $k$ sensors from a total of $n_s = 81$ candidate sensor locations uniformly distributed within a two-dimensional domain (cf. Figure~\ref{fig:v_field}). 

We choose $k = 10$ and $k=25$ sensors  out of $n_s = 81$. We use the RAF-OED approach as outlined in Section~\ref{sec:sensor_placement}.
We focus exclusively on RAF-OED, as the other methods would require storing all columns of the matrix 
$\A$, which becomes impractical for large-scale problems. 
For this application, we only compare against $5000$ random designs since an exhaustive search takes, e.g., $\binom{81}{10} \approx 1.8 \times 10^{12}$ combinations. Figure~\ref{fig:CS2D} displays the results of the comparison against random designs, and the sensor locations identified by RAF-OED are visualized in Figure~\ref{fig:sensor_locations}. As can be seen from the figure, the sensors selected by RAF-OED outperform the randomly selected designs for both cases $k=10$ and $25$). The results demonstrate the effectiveness of our proposed method in identifying near-optimal sensor placements.

\begin{figure}[!ht]
    \centering
    \begin{minipage}[b]{0.35\linewidth}
        \centering
        \includegraphics[width=0.9\linewidth]{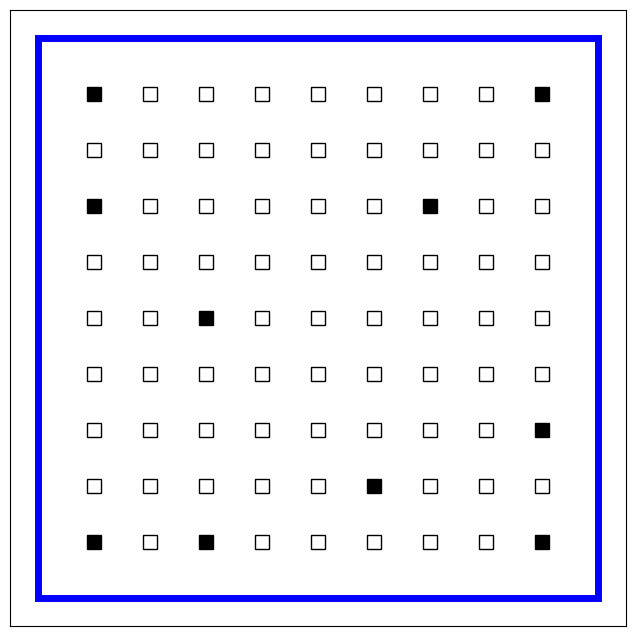}
        \subcaption*{Sensor locations $k=10$}
        \label{fig:best_10_sensors}
    \end{minipage}
    \hspace{2cm}
    \begin{minipage}[b]{0.35\linewidth}
        \centering
        \includegraphics[width=0.9\linewidth]{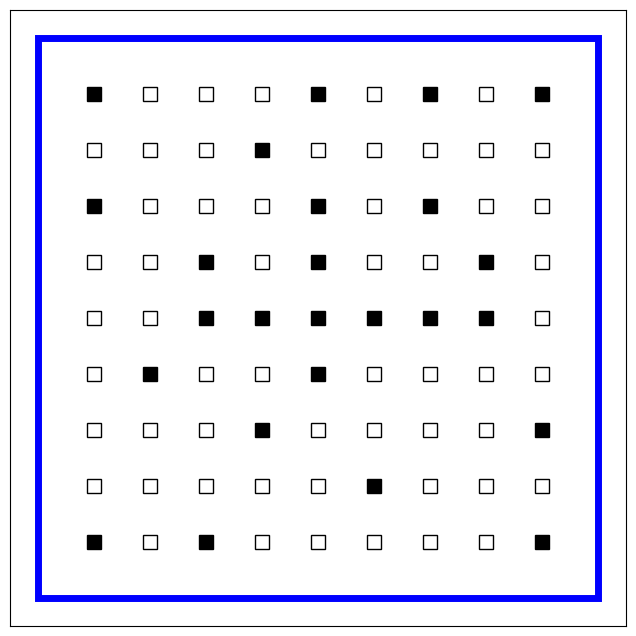}
        \subcaption*{Sensor locations $k=25$}
        \label{fig:best_25_sensors}
    \end{minipage}
    \caption{The sensor locations determined by RAF-OED for different $k$ values.}
    \label{fig:sensor_locations}
\end{figure}

\subsubsection*{Convergence of the EIG from WC to SC as model error decreases.}
We compare the EIG obtained from Weak-Constraint  and Strong-Constraint formulations as the model error decreases.  
We consider  $k=10$ sensors selected from $n_s=81$ candidates.  
Model error is simulated as $\bm\eta_\ell \sim \mathcal{N}(\bm 0, \bm Q_{\ell})$, where $\bm{Q}_{\ell} = q^2_\ell \I$ and 
$q_\ell := \alpha \tfrac{1}{\sqrt{n_{s}}} \| \Map{0}{\ell} \bm c^{\text{true}}_0 \|_2$ for $1 \le \ell \le n_T$.  
The parameter $\alpha$ controls the error magnitude.  
Table~\ref{tab:eig_wc_vs_sc} reports the SC-based EIG evaluated on the sensor sets selected by WC for different error levels.  
The SC-optimal design is independent of the model error. Since computing the exact SC-optimal design is infeasible, we approximate it using the GKS method, yielding
\[
\criteria^\text{SC}(\SSS{CSSP}^\text{SC}) \approx 170.32.
\]
In contrast, the WC design depends on the error size.  
As $\alpha$ decreases, the WC-based EIG values approach the SC reference value.  
This behavior is consistent with \Cref{proposition:bound_SCandWCoptCriteria,proposition:SCandWCoptCriteriaGap}.  

\begin{table}[htbp]
\centering
\begin{tabular}{lc}
\multicolumn{2}{c}{$ \criteria^\text{SC}(\mathbf{S}^\text{RGKS}_\text{WC})$} \\
\midrule
$\alpha = 5\%$ & $179.65$ \\
$\alpha = 3\%$ & $178.26$ \\
$\alpha = 2\%$ & $175.39$ \\
$\alpha = 1\%$ & $170.54$ \\
\end{tabular}
\caption{SC-based EIG values evaluated on sensor sets selected for the WC formulation.
$\mathbf{S}^\text{RGKS}_\text{WC}$ denotes the WC-selection obtained using the RAF-OED algorithm.}
\label{tab:eig_wc_vs_sc}
\end{table}

% 7. Conclusion

\section{Conclusions and Discussion}  
In this work, we introduced a novel approach to near-optimal sensor placement within the framework of OED, accounting for model errors through the WC4D-Var data assimilation technique.  

We developed an OED criterion for WC4D-Var based on the EIG and proved its convergence to the EIG corresponding SC4D-Var criterion as model error diminishes. We also demonstrated that the EIG-based optimal sensor selection performs at least as well as the SC4D-Var approach, with an explicit bound on the performance gap. We proposed several new reformulations of the criterion, and developed methods for computing them using stochastic trace estimators. 
We proposed algorithms based on column subset selection for sensor placement in time-dependent problems, along with a randomized, adjoint-free variant. 
These methods gave better optimal solutions compared to standard approaches, such as the greedy approach, while being more computationally efficient. Numerical experiments on one- and two-dimensional problems validated the effectiveness of both approaches, giving near-optimal solutions to the combinatorial sensor selection problem.

This work opens several promising directions for future research. A natural extension is to consider nonlinear dynamics. While our definition of the EIG can be used to define the OED criterion, it no longer has a closed form expression. There are many ways to address this computational challenge (see, e.g.,~\cite{Huan_Jagalur_Marzouk_2024}). Exploring other OED criteria such as A-optimality in the context of WC4D-Var is also of interest.

\bibliographystyle{abbrv}
\bibliography{refs}
% \begin{thebibliography}{99}

% \end{thebibliography}

\medskip
% The information below will be filled in by AIMS production staff.
Received xxxx 20xx; revised xxxx 20xx; early access xxxx 20xx.
\medskip

\end{document}